\documentclass[11pt]{article}
\usepackage{geometry}                % See geometry.pdf to learn the layout options. There are lots.
\usepackage{graphicx}
\usepackage{amsmath}
\usepackage{enumerate}
\usepackage{setspace}
\usepackage{amssymb}
\usepackage{amsmath, amsthm}
\usepackage{amsmath}
\usepackage{amsthm}
\usepackage{blkarray}
\usepackage{epstopdf}
\usepackage{footnote}
\newtheorem{thm}{Theorem}[section]

\newtheorem*{thm1}{Theorem A}

\newtheorem{definition}{Definition}[section]

\newtheorem*{remark}{Remark}
\newtheorem{lemma}[thm]{Lemma}
\newtheorem{proposition}[thm]{Proposition}

\begin{document}\large{
%%%%%%%%%%%%%%%%%%                            Title                                  %%%%%%%%%%%%%%%%%%%%%%%%%%%%%
\title{Determining system poles using row sequences of orthogonal Hermite-Pad\'e approximants}
\author{N. Bosuwan\thanks{The research of N. Bosuwan was supported by the research fund for DPST graduate with first placement from the Institute for the Promotion of Teaching Science and Technology (IPST 001/2557) and Faculty of Science, Mahidol University.}\,\,\,\footnote{Corresponding author.} and  G. L\'opez Lagomasino\thanks{The research of G. L\'opez Lagomasino was supported by research grant MTM2015-65888-C4-2-P from Ministerio de Econom\'ia y Competitividad}}}
\maketitle

\begin{abstract} Given a system of functions $\textup{\textbf{F}}=(F_1,\ldots,F_d),$ analytic on a neighborhood of some compact subset $E$ of the complex plane with simply connected complement, we define a sequence of vector rational functions with common denominator in terms of the orthogonal expansions of the components $F_k, k=1,\ldots,d,$ with respect to a sequence of orthonormal polynomials associated with a  measure $\mu$ whose support is contained in $E$. Such sequences of vector rational functions resemble row sequences of type II Hermite-Pad\'e approximants. Under appropriate assumptions on $\mu,$ we give necessary and sufficient conditions for the convergence with geometric rate of the common denominators of the sequence of vector rational functions so constructed. The exact rate of convergence of these denominators is provided and  the rate of convergence of the simultaneous approximants is estimated. It is shown that the common denominator of the approximants detect the location of the poles of the system of functions.
\end{abstract}

\vspace{0,7cm}


\vspace{0,5cm}\noindent

\section{Introduction}

Let $\textup{\textbf{F}}=(F_1,\ldots,F_d)$ be a system of $d$ formal or convergent Taylor expansions about the origin; that is, for each $i=1,\ldots,d,$ we have
\begin{equation}\label{Taylor}
F_i:=\sum_{n=0}^{\infty} f_{n,i} z^n, \quad \quad f_{n,i}\in \mathbb{C}.
\end{equation}

The following construction can be traced back to classical works of Ch. Hermite \cite{Her} and K. Mahler \cite{Mah}.

\begin{definition}\label{classical}\textup{
Let $\textup{\textbf{F}}=(F_1,\ldots,F_d)$ be a system of $d$ formal Taylor expansions as in \eqref{Taylor}. Fix a multi-index $\textup{\textbf{m}}=(m_1,\ldots,m_d)\in \mathbb{N}^d.$ Set
$$|\textup{\textbf{m}}|:=m_1+m_2+\ldots+m_d.$$ Then, for each $n \geq \max\{m_1,\ldots,m_d\},$ there exist polynomials $Q,$ $P_{i},$ $i=1,\ldots,d,$ such that
$$
\deg P_{i}\leq n-m_i, \quad i=1,\ldots,d, \quad \quad  \deg(Q)\leq |\textup{\textbf{m}}|, \quad Q\not\equiv 0,
$$
\begin{equation}\label{analogy}
Q(z)F_i(z)-P_i(z)=A_i z^{n+1}+\cdots.
\end{equation}
The vector of rational functions $\textup{\textbf{R}}_{n,\textup{\textbf{m}}}:= (P_1/Q,\ldots,P_d/Q)$ is called an \emph{$(n,\textup{\textbf{m}})$ (type II) Hermite-Pad\'e approximant of $\textup{\textbf{F}}.$}
}
\end{definition}

Alternatively, one can solve the following problem
\begin{proposition}\label{classical1}\textup{
Given $\textup{\textbf{F}}=(F_1,\ldots,F_d)$  and  $\textup{\textbf{m}}=(m_1,\ldots,m_d)\in \mathbb{N}^d,$ find polynomials $Q,$ $P_{k,i},$ $k=0,1,\ldots,m_i-1,$ $i=1,\ldots,d,$ such that for all $i=1,2,\ldots,d,$
$$
\deg P_{k,i}\leq n-1, \quad k=0,\ldots,m_i-1, \quad \quad  \deg(Q)\leq |\textup{\textbf{m}}|, \quad Q\not\equiv 0,
$$
\begin{equation}\label{analogy1}
Q(z)z^kF_i(z)-P_{k,i}(z)=A_i z^{n+1}+\cdots, \qquad k=0,\ldots,m_i-1.
\end{equation}
}
\end{proposition}

It is easy to verify that the system of homogeneous linear equations to be solved in order to find the polynomial $Q$ in \eqref{analogy} and \eqref{analogy1} is the same. Once $Q$ is found the polynomial $P_{i}$ in Definition \ref{classical} and $P_{k,i}$ in Proposition \ref{classical1} are uniquely determined. In this sense, Definition \ref{classical} and Proposition \ref{classical1} solve the same problem. However, we wish to point out that different solutions for $Q$ may produce, in general, different $(n,\textup{\textbf{m}})$ Hermite-Pad\'e approximants of ${\bf F}$. In the sequel, given $(n,\textup{\textbf{m}}),$ one particular solution is taken. For that solution, we write
$$\textup{\textbf{R}}_{n,\textup{\textbf{m}}}:=(R_{n,\textup{\textbf{m}},1},\ldots,R_{n,\textup{\textbf{m}},d})=(P_{n,\textup{\textbf{m}},1},\ldots,P_{n,\textup{\textbf{m}},d})/Q_{n,|\textup{\textbf{m}}|},$$
when $Q_{n,|\textup{\textbf{m}}|}$ is monic and has no common zero simultaneously with all the $P_{n,\textup{\textbf{m}},i}.$

Most papers devoted to Hermite-Pad\'e approximation deal with diagonal or near diagonal sequences (when $|{\bf m}| \approx n$) and their application in several areas (such as multiple orthogonal polynomials, number theory, random matrices, brownian motions, Toda lattices, to name a few). Less attention has been paid to the theory related with row sequences, when $\bf m$ remains fixed independent of $n$.

The first significant contribution on the convergence of row sequences of Hermite-Pad\'e approximation is due to Graves-Morris and Saff \cite{MorrisSaff} (see also \cite{GS2} and \cite{GS3})  where a Montessus de Ballore type theorem \cite{Mon} was proved, under the assumption of polewise independence of the system of fucntions. This concept was introduced by the authors in the same paper. Recently, Cacoq, de la Calle, and L\'opez \cite{CacoqYsernLopezIncom} improved that result in several directions; namely, improving the estimate on the rate of convergence and weakening the assumption of polewise independence. In its final form, in \cite[Theorem 1.4 and Theorem 3.7]{CacoqYsernLopez} the authors prove an analogue of the Montessus de Ballore-Gonchar theorem. To state that result, we need to introduce  some concepts and notation.

Let $\boldsymbol \Omega:=(\Omega_1,\ldots,\Omega_d)$ be a system of domains such that, for each $i=1,\ldots,d,$ $F_i$ is meromorphic in $\Omega_i.$ We say that the point $\xi\in \Omega_i$ is a pole of $\textup{\textbf{F}}$ in $\boldsymbol \Omega$ of order $\tau$ if there exists an index $i\in \{1,\ldots,d\}$ such that $\xi \in \Omega_i$ and it is a pole of $F_i$ of order $\tau,$ and for $j\not=i$ either $\xi$ is a pole of $F_j$ of order less than or equal to $\tau$ or $\xi\not\in \Omega_j.$ When $\boldsymbol \Omega=(\Omega,\ldots,\Omega),$ we say that $\xi$ is a pole of $\textup{\textbf{F}}$ in $\Omega.$

Denote by $$\mathbb{B}_{R}:=\{z\in \mathbb{C}:|z|<R\}$$ the disk centered at the origin of radius $R.$
Let $R_{0}(\textup{\textbf{F}})$ be the radius of the largest disk $\mathbb{B}_{R_{0}(\textup{\textbf{F}})}$ to which all the expansions $F_i,$ $i=1,\ldots,d$ can be extended analytically. If $R_{0}(\textup{\textbf{F}})=0,$ we take $\mathbb{B}_{R_m(\textup{\textbf{F}})}=\emptyset,$ $m\geq 0;$ otherwise, $R_{m}(\textup{\textbf{F}})$ is the radius of the largest disk $\mathbb{B}_{R_{m}(\textup{\textbf{F}})}$ centered at the origin to which all the analytic elements $(F_i, \mathbb{B}_{R_{0}(F_i)})$ can be extended so that $\textup{\textbf{F}}$ has at most $m$ poles counting multiplicities. Denote by  $\mathbb{N}$   the set of all positive integers.

Let us define the concept of system pole for a vector of $d$ formal Taylor expansions  $\textup{\textbf{F}}$ as in \eqref{Taylor}.
\begin{definition}\label{SP} \textup{ Given  $\textup{\textbf{F}}=(F_1,F_2,\ldots,F_d)$  and $\textup{\textbf{m}}=(m_1,m_2,\ldots,m_d)\in \mathbb{N}^d $, we say that $\xi\in \mathbb{C}\setminus\{0\}$ is a \emph{system pole of order $\tau$ of $\textup{\textbf{F}}$ with respect to $\textup{\textbf{m}}$} if $\tau$ is the largest positive integer such that for each $t=1,2,\ldots,\tau,$ there exists at least one polynomial combination of the form
\begin{equation}\label{polycom11}
\sum_{i=1}^d v_i F_i,\quad \quad \deg v_i<m_i,\quad \quad i=1,2,\ldots,d,
\end{equation}
which is holomorphic on a neighborhood of $\overline{\mathbb{B}}_{|\xi|}$ except for a pole at $z=\xi$ of exact order $t.$ }
\end{definition}

To each system pole $\xi$ of $\textup{\textbf{F}}$ with respect to $\textup{\textbf{m}}$ we associate several characteristic values. Let $\tau$ be the order of $\xi$ as a system pole of $\textup{\textbf{F}}$. For each $t = 1,\ldots, \tau,$ denote by $r_{\xi,t}(\textup{\textbf{F}},\textup{\textbf{m}})$ the largest of all the numbers $R_{t}(g)$ (the radius of the largest disk containing at most $t$ poles of $g$), where $g$ is a polynomial combination of type \eqref{polycom11} that is analytic on a neighborhood of $\mathbb{B}_{|\xi|}$ except for a pole at $z = \xi$ of order $t$. Then,
$$R_{\xi,t}(\textup{{\textbf{F}}},\textup{\textbf{m}}):=\min_{k=1,\ldots,t} r_{\xi,k} (\textup{\textbf{F}},\textup{\textbf{m}}),$$
$$R_{\xi}(\textup{\textbf{F}},\textup{\textbf{m}}):=R_{\xi,\tau}(\textup{{\textbf{F}}},\textup{\textbf{m}})=\min_{k=1,\ldots,\tau} r_{\xi,k} (\textup{\textbf{F}},\textup{\textbf{m}}).$$

By $Q_{\textup{\textbf{m}}}^{\textup{\textbf{F}}},$ we denote the monic polynomial whose zeros are the system poles of $\textup{\textbf{F}}$ with respect to $\textup{\textbf{m}}$ taking account of their order. The set of distinct zeros of $Q_{\textup{\textbf{m}}}^{\textup{\textbf{F}}}$ is denoted by $\mathcal{P}(\textup{\textbf{F}},\textup{\textbf{m}}).$

The following theorem (see  \cite[Theorem 1.4 and Theorem 3.7]{CacoqYsernLopez}) is an analogue of the Montessus de Ballore-Gonchar theorem.
\begin{thm1}
Let $\textup{\textbf{F}}$ be a vector of formal Taylor expansions at the origin and fix a multi-index $\textup{\textbf{m}}\in \mathbb{N}^d.$ The following two assertions are equivalent:
\begin{enumerate}
\item [(a)] $R_{0}(\textup{\textbf{F}})>0$ and $\textup{\textbf{F}}$ has exactly $|\textup{\textbf{m}}|$ system poles with respect to $\textup{\textbf{m}}$ counting multiplicities.
\item [(b)] The denominators $Q_{n,|\textup{\textbf{m}}|},$ $n \geq |\textup{\textbf{m}}|,$ of the Hermite-Pad\'e approximants of $\textup{\textbf{F}}$ are uniquely determined for all sufficiently large $n,$ and there exists a polynomial $Q_{|\textup{\textbf{m}}|}$ of degree $|\textup{\textbf{m}}|,$ $Q_{|\textup{\textbf{m}}|}(0)\not=0,$ such that
$$\limsup_{n \rightarrow \infty} \|Q_{|\textup{\textbf{m}}|}-Q_{n,|\textup{\textbf{m}}|}\|^{1/n}=\theta<1,$$
where $\|\cdot\|$ denotes the coefficient norm in the space of polynomials.
\end{enumerate}
Moreover, if either (a) or (b) takes place, then $Q_{|\textup{\textbf{m}}|}\equiv Q_{\textup{\textbf{m}}}^{\textup{\textbf{F}}},$ and
$$\theta=\max\left\{ \frac{|\xi|}{R_{\xi}(\textup{\textbf{F}},\textup{\textbf{m}})} :  \mathcal{P}(\textup{\textbf{F}},\textup{\textbf{m}}) \right\}.$$
 \end{thm1}

An exact expression for the rate and region of convergence of $R_{n,\textup{\textbf{m}},i}$ to $F_i$ as $n \to \infty$ is also given, see \cite[Theorem 3.7]{CacoqYsernLopez}.

The object of this paper is to give a similar result when Taylor expansions are replaced by orthogonal ones in the sense we will describe below.

Let $E$ be an infinite compact subset of the complex plane $\mathbb{C}$ such that $\overline{\mathbb{C}}\setminus E$ is simply connected.
Let $\mu$ be a finite positive Borel measure with infinite support $\mbox{supp}(\mu)$ contained in $E$. We write $\mu \in \mathcal{M}(E)$ and define the associated inner product,
$$\langle g,h \rangle_{\mu}:=\int g(\zeta) \overline{h(\zeta)} d\mu(\zeta), \quad g,h \in L_2(\mu).$$
Let
$$p_{n}(z):=\kappa_n z^n+\cdots, \quad  \kappa_n>0,\quad n=0,1,2,\ldots,$$ be the orthonormal polynomial of degree $n$ with respect to $\mu$ with positive leading coefficient; that is, $\langle p_n, p_m \rangle_{\mu}=\delta_{n,m}.$  Denote by $\mathcal{H}(E)$ the space of all functions holomorphic in some neighborhood of $E.$ We define
$$\mathcal{H}(E)^d:=\{(F_1,F_2,\ldots,F_d):  F_i\in \mathcal{H}(E),  i=1,2,\ldots,d \}.$$

A natural way of extending the notion of Hermite-Pad\'e approximation as given by Definition \ref{classical} to the present setting is the following.

\begin{definition}\label{simu11111}\textup{ Let $\textup{\textbf{F}}=(F_1,F_2,\ldots,F_d)\in \mathcal{H}(E)^d$  and $\mu\in \mathcal{M}(E).$ Fix a multi-index $\textup{\textbf{m}}=(m_1,m_2,\ldots,m_d)\in \mathbb{N}^d.$  Set $|\textup{\textbf{m}}|=m_1+m_2+\ldots+m_d.$ Then, for each $n\geq \max\{m_1,m_2,\ldots,m_d\},$ there exist polynomials $\tilde{Q}_{n,\textup{\textbf{m}}}^{\mu}$ and $ \tilde{P}_{n,\textup{\textbf{m}},i}^{\mu},$ $i=1,2,\ldots,d$ such that
$$\deg(\tilde{P}_{n,\textup{\textbf{m}},i}^{\mu})\leq n-m_i, \quad \deg(\tilde{Q}_{n,|\textup{\textbf{m}}|}^{\mu})\leq |\textup{\textbf{m}}|, \quad \tilde{Q}_{n,|\textup{\textbf{m}}|}^{\mu}\not\equiv 0,$$
$$\langle \tilde{Q}_{n,|\textup{\textbf{m}}|}^{\mu} F_i-\tilde{P}_{n,\textup{\textbf{m}},i}^{\mu},\, p_{j} \rangle_{\mu}=0, \quad  j=0,1,\ldots,n,$$
for all $i=1,2,\ldots,d.$
The vector rational function
$$\tilde{\textup{\textbf{ R}}}_{n,\textup{\textbf{m}}}^{\mu}:=(\tilde{R}_{n,\textup{\textbf{m}},1}^{\mu}, \ldots,\tilde{R}_{n,\textup{\textbf{m}},d}^{\mu})=(\tilde{P}_{n,\textup{\textbf{m}},1}^{\mu}/\tilde{Q}_{n,|\textup{\textbf{m}}|}^{\mu} ,\ldots,\tilde{P}_{n,\textup{\textbf{m}},d}^{\mu}/\tilde{Q}_{n,|\textup{\textbf{m}}|}^{\mu})$$ is called an \emph{$(n, \textup{\textbf{m}})$  simultaneous Fourier-Pad\'e approximant of $\textup{\textbf{F}}$ with respect to $\mu.$}}
\end{definition}

The convergence of simultaneous Fourier-Pad\'e approximation was first investigated in  \cite{CacoqLopez} for the case when $E=\{z\in \mathbb{C}:|z|\leq 1\}$ and the support of  $\mu$ is contained in the unit circle (see also \cite{bosuwan,bosuwan1} for the case when $d=1$ and $E$ is a general compact set). The results obtained in \cite{CacoqLopez} are not very promising for several reasons. The restrictions imposed on the measure are stronger than what is to be expected, the extension to the case of measures supported on more general compact sets of the complex plane does not appear to be very plausible, the authors only obtain direct results ((a) implies (b))  using the assumption of polewise independence introduced in \cite{MorrisSaff} which already in the classical case of Hermite-Pad\'e approximation does not lead to an inverse statement ((b) implies (a)) in the theorem.

It is easy to check, when $E=\{z\in \mathbb{C}:|z|\leq 1\}$ and $d\mu=d\theta/2\pi$ on the unit circle, that the concepts of Hermite-Pad\'e approximation and that of simultaneous Fourier-Pad\'e approximation  with respect to $d\theta/2\pi$ of $\textup{\textbf{F}}$ coincide. By the same token, for this special case the simultaneous Fourier-Pad\'e approximants verify Proposition \ref{classical1}. However, for other measures the analogues of Definition \ref{classical} and Proposition \ref{classical1} lead to different homogeneous linear systems of equations. It turns out, that the correct way to extend the notion of Hermite-Pad\'e approximation to the case of vector orthogonal expansions in order to obtain direct and inverse type results is through Proposition \ref{classical1}. So, we propose the following definition.

\begin{definition}\label{simu}\textup{ Let $\textup{\textbf{F}}=(F_1,F_2,\ldots,F_d)\in \mathcal{H}(E)^d$  and $\mu\in \mathcal{M}(E).$ Fix a multi-index $\textup{\textbf{m}}=(m_1,m_2,\ldots,m_d)\in \mathbb{N}^d$ and  $n\in \mathbb{N}.$ Then, there exist polynomials $Q_{n,|\textup{\textbf{m}}|}^{\mu},$ $P_{n,\textup{\textbf{m}},k,i}^{\mu},$ $k=0,1,\ldots,m_i-1, i=1,2,\ldots,d$ such that for all $ i=1,2,\ldots,d,$
\begin{equation}\label{simu1}
\deg (P_{n,\textup{\textbf{m}},k,i}^{\mu})\leq n-1, \quad k=0,1,\ldots,m_i-1, \quad   \deg(Q_{n,|\textup{\textbf{m}}|}^{\mu})\leq |\textup{\textbf{m}}|, \quad Q_{n,|\textup{\textbf{m}}|}^{\mu}\not\equiv 0,
\end{equation}
\begin{equation}\label{simu2}
\langle Q_{n,|\textup{\textbf{m}}|}^{\mu} z^{k} F_i-P_{n,\textup{\textbf{m}},k,i}^{\mu},\, p_{j} \rangle_{\mu}=0,  \quad \quad k=0,1,\ldots,m_i-1 \quad \quad j=0,1,\ldots,n.
\end{equation}  The vector  rational function
$$\textup{\textbf{R}}^{\mu}_{n,\textup{\textbf{m}}}:=(R_{n,\textup{\textbf{m}},1}^{\mu},\ldots,R_{n,\textup{\textbf{m}},d}^{\mu})=(P_{n,\textup{\textbf{m},0,1}}^{\mu}, \ldots, P_{n,\textup{\textbf{m},0,d}}^{\mu})/Q_{n,|\textup{\textbf{m}}|}^{\mu}$$ is called an \emph{$(n,\textup{\textbf{m}})$ orthogonal Hermite-Pad\'e approximant of $\textup{\textbf{F}}$ with respect to $\mu.$} }
\end{definition}
Clearly,
\begin{equation}
\label{simu3}
\langle Q_{n,|\textup{\textbf{m}}|}^{\mu} z^{k} F_i,\, p_{n} \rangle_{\mu}=0,\quad \quad i=1,\ldots,d,  \quad \quad k=0,1,\ldots,m_i-1.
\end{equation}
Since $Q_{n,|\textup{\textbf{m}}|}^{\mu}\not\equiv 0,$ we normalize it to have leading coefficient equal to $1.$ We call $Q_{n,|\textup{\textbf{m}}|}^{\mu}$ \emph{a denominator of an $(n,\textup{\textbf{m}})$ orthogonal Hermite-Pad\'e approximant of $\textup{\textbf{F}}$ with respect to $\mu$}.

From \eqref{simu1}-\eqref{simu2} it is obvious that the polynomials $P_{n,\textup{\textbf{m}},k,i}^{\mu}$ are uniquely determined once $Q_{n,|\textup{\textbf{m}}|}^{\mu}$ is found as a solution of the homogeneous linear system of $|{\bf m}|$ equations on the $|{\bf m}|+1$ unknown coefficients of $Q_{n,|\textup{\textbf{m}}|}^{\mu}$ resulting from \eqref{simu3}. Therefore, for any pair $(n,\textup{\textbf{m}})\in \mathbb{N}\times \mathbb{N}^d,$ a rational function $\textup{\textbf{R}}^{\mu}_{n,\textup{\textbf{m}}}$ always exists. However, in general, $\textup{\textbf{R}}^{\mu}_{n,\textup{\textbf{m}}}$ may not be unique. In this paper, we will restrict our attention to  orthogonal Hermite-Pad\'e approximants (as in Definition \ref{simu}).

Given $E\subset \mathbb{C}$ with $\overline{C}\setminus E$ simply connected, there exists a unique (exterior) conformal mapping $\Phi$ from $\overline{\mathbb{C}}\setminus {E}$ onto $\overline{\mathbb{C}}\setminus \{w\in \mathbb{C}: |w|\leq 1\}$ satisfying $\Phi(\infty)=\infty$ and $\Phi'(\infty)>0.$
For each $\rho>1,$  we introduce
 $$\Gamma_{\rho}:=\{z\in \mathbb{C}: |\Phi(z)|=\rho\}, \quad \quad \mbox{and} \quad \quad D_{\rho}:=E\cup \{z\in \mathbb{C}: |\Phi(z)|<\rho\},$$
 as the \emph{level curve of index $\rho$} and the \emph{canonical domain of index $\rho$}, respectively.
Let $\rho_0(\textup{\textbf{F}})$ be equal to the index $\rho$ of the largest canonical domain $D_{\rho}$ to which all $F_i,$ $i=1,\ldots,d,$ can be extended as holomorphic functions. Moreover, $\rho_{m}(\textup{\textbf{F}})$ is the index of the largest canonical domain $D_\rho$ to which all $F_i,$ $i=1,\ldots,d$ can be extended so that $\textup{\textbf{F}}$ has at most $m$ poles counting multiplicities.

In analogy with Definition \ref{SP} we give
\begin{definition}\textup{ Given $\textup{\textbf{F}}=(F_1,F_2,\ldots,F_d)\in\mathcal{H}(E)^d$ and $\textup{\textbf{m}}=(m_1,m_2,\ldots,m_d)\in \mathbb{N}^d $, we say that $\xi\in \mathbb{C}$ is a \emph{system pole of order $\tau$ of $\textup{\textbf{F}}$ with respect to $\textup{\textbf{m}}$} if $\tau$ is the largest positive integer such that for each $t=1,2,\ldots,\tau,$ there exists at least one polynomial combination of the form
\begin{equation}\label{polycom}
\sum_{i=1}^d v_i F_i,\quad \quad \deg v_i<m_i,\quad \quad i=1,2,\ldots,d,
\end{equation}
which is holomorphic on a neighborhood of $\overline{D}_{|\Phi(\xi)|}$ except for a pole at $z=\xi$ of exact order $t.$ }
\end{definition}

As above, let $E$ be a compact set such that $\overline{\mathbb{C}} \setminus E$ is simply connected and $\mu \in \mathcal{M}(E)$. Let $p_n(z)$ be the $n$-th orthonormal polynomial of $\mu$ with positive leading coefficient $\kappa_n$. The measure $\mu$ is said to be regular, and we write $\mu \in {\mbox{\bf Reg}}$, if
\[ \lim_{n\to \infty} \kappa_n^{1/n} = \frac{1}{\mbox{cap}(\mbox{supp}(\mu))},\]
where $\mbox{cap}(\mbox{supp}(\mu))$ denotes the logarithmic capacity of $\mbox{supp}(\mu)$ (for the definition of regular measures and its different defining properties see \cite[Theorem 3.1.1]{totik}). We are interested in regular measures for which $\mbox{cap}(\mbox{supp}(\mu)) = \mbox{cap}(E)$ and for them we write $\mu \in {\mbox{\bf Reg}}(E)$. Since $E\subset \mathbb{C}$ is a compact set such that $\overline{\mathbb{C}}\setminus E$ is simply connected, it is well known, see \cite[Theorem 3, p. 314]{Gol}, that
\[\mbox{cap}(E) = 1/|\Phi'(\infty)|,\]
and regularity is equivalent in this case (see again \cite[Theorem 3.1.1]{totik}) to
\begin{equation}
\label{asintlog}
  \lim_{n \to \infty} |p_n(z)|^{1/n} = |\Phi(z)|,
\end{equation}
uniformly inside $\mathbb{C} \setminus \mbox{Co}(E)$, where $\mbox{Co}(E)$ denotes the convex hull of $E$. Here, and in what follows, the phrase ``uniformly inside a domain" means ``uniformly on each compact subset of the indicated domain". Now, if $E$ itself is convex then
\eqref{asintlog} takes place in the complement of $E$. When $E \setminus \mbox{Co}(E) \neq \emptyset,$ there may be $o(n)$ zeros of $p_n$ wandering around this complement which affect the $n$-th root asymptotic in that region. We say that $\mu \in {\mbox{\bf Reg}}_1(E)$ when \eqref{asintlog} takes place uniformly on compact subsets of $\mathbb{C} \setminus E$.

Let us introduce the second type function
\[s_n(z) := \int \frac{\overline{p_n(\zeta)}}{z - \zeta} d\mu(\zeta), \qquad z\in \overline{\mathbb{C}}\setminus \mbox{supp}(\mu)\]
From orthogonality it readily follows that
\[p_n(z)s_n(z) := \int \frac{|p_n(\zeta)|^2}{z - \zeta} d\mu(\zeta).\]
It is easy to check that for any compact subset $K \subset \overline{\mathbb{C}}\setminus \mbox{Co}(E)$ there exist positive constants $C_1(K), C_2(K),$ independent of $n,$ such that
\[C_1(K)  \leq \left|\int \frac{|p_n(\zeta)|^2}{z - \zeta} d\mu(\zeta)\right| \leq C_2(K), \qquad z \in K,\]
so that if $\mu \in {\mbox{\bf Reg}}(E)$ then
\begin{equation}
\label{asintlog2}
  \lim_{n \to \infty} |s_n(z)|^{1/n} = |\Phi(z)|^{-1},
\end{equation}
uniformly inside $\mathbb{C} \setminus \mbox{Co}(E)$. We say that $\mu \in {\mbox{\bf Reg}}_2(E)$ when \eqref{asintlog2} takes place uniformly inside $\mathbb{C} \setminus E$. Consequently, when $E$ is convex then ${\mbox{\bf Reg}}(E) = {\mbox{\bf Reg}}_1(E)={\mbox{\bf Reg}}_2(E)$. When both \eqref{asintlog} and \eqref{asintlog2} hold uniformly inside $\overline{\mathbb{C}} \setminus E$ we write $\mu \in {\mbox{\bf Reg}}_{1,2}(E)$.

Fix $0 \leq m \leq n$. From the extremal properties of monic orthogonal polynomials $p_n = \kappa_n P_n$ in the $L_2$ norm, we have
\[ \frac{1}{\kappa_n^2} = \int |P_n(z)|^2 d\mu(z) \leq \int |z^{m}P_{n-m}(z)|^2 d\mu(z) \leq \|z\|_E^{2m} \int |P_{n-m}(z)|^2 d\mu(z) = \frac{\|z \|_E^{2m}}{\kappa_{n-m}^2}.\]
Whence
\[\frac{\kappa_{n-m}}{\kappa_n} \leq \|z\|^m_E, \qquad n \geq m,\]
where $\|\cdot \|_{E}$ denotes the sup-norm on $E.$
We need an analogous uniform bound with respect to $n$ from below. We say that $\mu \in \mbox{\bf Reg}_{1,2}^m(E)$ if it is in $\mbox{\bf Reg}_{1,2}(E)$ and there exists a positive constant $c$ such that
\begin{equation}
\label{below}
\frac{\kappa_{n-m}}{\kappa_n} \geq c, \qquad n \geq n_0.
\end{equation}

\indent All measures in the complex plane whose orthonormal polynomials verify strong asymptotic  are in $\mbox{\bf Reg}_{1,2}^m(E)$, see \cite{Widom1969}.
Other classes of measures related with ratio asymptotics of orthogonal polynomials, which are contained in $\textup{\textbf{Reg}}_{1,2}^{m}(E)$ may be found in \cite{Be,bosuwan,bosuwan1}. Unfortunately, there are no results of general character, in terms of the analytic properties of the measure, describing the measures in the complex plane whose sequence of orthonormal polynomials have ratio asymptotic, except when $E$ is a segment of the real line, the unit circle, or an arc of the unit circle (see, for example, \cite{Be}, \cite{Rakhmanov}, \cite{Rakhmanov2}).

Let $\tau$ be the order of $\xi$ as a system pole of $\textup{\textbf{F}}.$ For each $t=1,\ldots,\tau,$ denote by $\rho_{\xi,t}(\textup{\textbf{F}},\textup{\textbf{m}})$ the largest of all the numbers $\rho_t(G)$ (the index of the largest canonical domain containing at most $t$ poles of $G$), where $G$ is a polynomial combination of type \eqref{polycom} that is holomorphic on a neighborhood of $\overline{D}_{|\Phi(\xi)|}$ except for a pole at $z=\xi$ of order $t.$ Then, we define
$$\boldsymbol\rho_{\xi,t} (\textup{\textbf{F}}, \textup{\textbf{m}}):=\min_{k=1,\ldots,t} \rho_{\xi, k}(\textup{\textbf{F}}, \textup{\textbf{m}}),$$
$$\boldsymbol\rho_{\xi} (\textup{\textbf{F}}, \textup{\textbf{m}}):=\boldsymbol\rho_{\xi,\tau} (\textup{\textbf{F}}, \textup{\textbf{m}})=\min_{t=1,\ldots, \tau} \rho_{\xi, t} (\textup{\textbf{F}},\textup{\textbf{m}}).$$

Fix $i\in\{1,\ldots,d\}$ and $k\in\{0,1,\ldots,m_i-1\}.$ Let $D_{{i,k}}(\textup{\textbf{F}},\textup{\textbf{m}})$ be the largest canonical domain in which all the poles of $z^kF_{i}$ are system poles of $\textup{\textbf{F}}$ with respect to $\textup{\textbf{m}},$ their order as poles of $z^{k}F_i$ does not exceed their order as system poles, and $z^{k}F_i$ has no other singularity. By $\boldsymbol\rho_{i,k}(\textup{\textbf{F}},\textup{\textbf{m}}),$ we denote the index of this canonical domain. Let $\xi_{1},\ldots,\xi_N$ be the poles of $z^{k}F_i$ in $D_{i,k}(\textup{\textbf{F}},\textup{\textbf{m}}).$  For each $j=1,\ldots,N,$ let $\hat{\tau}_j$ be the order of $\xi_j$ as pole of $z^{k}F_i$ and $\tau_j$ its order as a system pole. By assumption, $\hat{\tau}_j\leq \tau_j.$ Set
$$\boldsymbol\rho_{i,k}^{*}(\textup{\textbf{F}},\textup{\textbf{m}}):=\min\{ \boldsymbol\rho_{i,k}(\textup{\textbf{F}},\textup{\textbf{m}}),\min_{j=1,\ldots,N} \boldsymbol\rho_{\xi_j,\hat{\tau}_j}(\textup{\textbf{F}},\textup{\textbf{m}})\}$$
and let $D_{i,k}^{*}(\textup{\textbf{F}},\textup{\textbf{m}})$ be the canonical domain with this index.

The main result of the paper is the following.

\begin{thm} \label{complete} Let $\textup{\textbf{F}}=(F_1,F_2,\ldots,F_d)\in \mathcal{H}(E)^d,$  $\textup{\textbf{m}}\in \mathbb{N}^d$ be a fixed multi-index, and  $\mu\in \textup{\textbf{Reg}}_{1,2}^{|\textup{\textbf{m}}|}(E).$  Then, the following two assertions are equivalent:
\begin{enumerate}
\item [(a)] $\textup{\textbf{F}}$ has exactly $|\textup{\textbf{m}}|$ system poles with respect to $\textup{\textbf{m}}$ counting multiplicities.
\item [(b)] The polynomials $Q_{n,|\textup{\textbf{m}}|}^{\mu}$ of $\textup{\textbf{F}}$ are uniquely determined for all sufficiently large $n,$ and there exists a polynomial $Q_{|\textup{\textbf{m}}|}$ of degree $|\textup{\textbf{m}}|$ such that
$$\limsup_{n \rightarrow \infty} \|Q_{n,|\textup{\textbf{m}}|}^{\mu}-Q_{|\textup{\textbf{m}}|}\|^{1/n}=\theta<1.$$
\end{enumerate}
Moreover, if either (a) or (b) takes place, then $Q_{|\textup{\textbf{m}}|}=Q_{\textup{\textbf{m}}}^{\textup{\textbf{F}}},$
$$\theta=\max \left\{\frac{|\Phi(\xi)|}{\boldsymbol\rho_{\xi}(\textup{\textbf{F}},\textup{\textbf{m}})} :\xi\in\mathcal{P}(\textup{\textbf{F}},\textup{\textbf{m}})\right\},$$ and
for any compact subset $K$ of $D_{i,0}^{*}(\textup{\textbf{F}},\textup{\textbf{m}})\setminus \mathcal{P}(\textup{\textbf{F}},\textup{\textbf{m}}),$
\begin{equation*}
\limsup_{n \rightarrow \infty} \|R_{n,\textup{\textbf{m}},i}^{\mu}-F_i\|_{K}^{1/n}\leq \frac{\|\Phi\|_K}{\boldsymbol \rho_{i,0}^{*}({\textup{\textbf{F}},\textup{\textbf{m}}})},
\end{equation*}
where $\|\cdot \|_{K}$ denotes the sup-norm on $K$ and if $K\subset E,$ then $\|\Phi\|_K$ is replaced by $1.$
\end{thm}

Theorem \ref{complete} is a direct consequence of Theorems \ref{thm1.4} and \ref{inverse}. In Theorem \ref{thm1.4} we prove that (a) implies (b) and estimate the rate of convergence of the orthogonal Hermite-Pad\'e approximants. Theorem \ref{inverse} contains the inverse assertion (b) implies (a). This is done in Sections 2 and 3, respectively.

\section{The direct statements}

\subsection{On the convergence of orthogonal expansions}\label{Discussionaboutasymptotics}

First of all, let us discuss some properties of orthogonal polynomial expansions of holomorphic functions.
Let $\mu \in \textup{\textbf{Reg}}_1(E)$. The $n$-th Fourier coefficient of $G\in \mathcal{H}(E)$ with respect to $p_n$ is given by
\begin{equation*}\label{Fourierco}
[G]_n:=\langle G,\,p_n\rangle_\mu =\int G(z) \overline{p_n(z)}d\mu(z).
\end{equation*}
The following lemma (see, e.g., Theorem 6.6.1 in \cite{totik}) is well known but we could not find an appropriate reference in our setting so we sketch a proof.
\begin{lemma}\label{expan} Let $G\in \mathcal{H}(E)$ and $\mu\in \textup{\textbf{Reg}}_1(E)$. Then,
\begin{equation}\label{defofrhomF}
\rho_0(G)=\left(\limsup_{n \rightarrow \infty} |[G]_n|^{1/n} \right)^{-1}.
\end{equation}
Moroever, the series $\sum_{n=0}^{\infty} [G]_n p_n(z)$ converges absolutely and uniformly inside  ${D}_{\rho_{0}(G)}$ to $G(z)$, and diverges pointwise for all $z\in \mathbb{C}\setminus \overline{D_{\rho_0(G)}}.$
\end{lemma}

\begin{proof}
[Proof of Lemma \ref{expan}]
The absolute and uniform convergence of the series on compact subsets of ${D}_{\rho_{0}(G)}$ is carried out using \eqref{asintlog} in the same way as similar statements for Taylor series. Let $G_1$ be the uniform limit. Obviously, $G_1 \in \mathcal{H}({D}_{\rho_{0}(G)})$. The pointwise divergence in the complement of ${D}_{\rho_{0}(G)}$ is also obtained as for Taylor series.

Since $\mathbb{C} \setminus E$ is connected by Mergelyan's theorem, there exists a sequence of polynomials $(g_n), n \in \mathbb{N}, \deg(p_n) = n$ such that
\[\lim_{n\to \infty} \|G- g_n\|_E = 0,\]
where $\|\cdot\|_E$ denotes the uniform norm on $E$.
Now
\[0 \leq \lim_{n\to \infty}\left(\int |(G - g_n)(x)|^2 d\mu(x)\right)^{1/2} \leq \mu(E)^{1/2}\lim_{n\to\infty}\|G-g_n\|_E = 0.\]
Therefore the partial sums of the Fourier expansion converge to $G$ in $L_2(\mu)$. So there is a subsequence of the partial sums that convergence $\mu$ almost everywhere of $E$ to $G$. Thus, $G = G_1$ $\mu$ almost everywhere and consequently $G\equiv G_1$.
\end{proof}

As a consequence of Lemma \ref{expan},
if  $\mu\in \textup{\textbf{Reg}}_1(E)$ and $\textup{\textbf{F}}=(F_1,F_2,\ldots,F_d)\in \mathcal{H}(E)^d$, then for each $i=1,2,\ldots,d$ and $k=0,1,\ldots,m_i-1$ fixed
\begin{equation}\label{usethisasdef}
z^{k}Q_{n,|\textup{\textbf{m}}|}^{\mu}(z)F_i(z)-P_{n,\textup{\textbf{m}},k,i}^{\mu}(z)=\sum_{\beta=n+1}^{\infty} [z^{k} Q_{n,|\textup{\textbf{m}}|}^{\mu}F_i]_{\beta}\,p_{\beta}(z), \quad \quad z\in D_{\rho_{0}(F_i)},
\end{equation}
and $P_{n,\textup{\textbf{m}},k,i}^{\mu}=\sum_{\beta=0}^{n-1} [ z^{k} Q_{n,|\textup{\textbf{m}}|}^{\mu} F_i]_{\beta}\,p_\beta$ is uniquely determined by $Q_{n,|\textup{\textbf{m}}|}^{\mu}.$

A simple relation used frequently in this paper is contained in
\begin{lemma}\label{usealot1}
Let $G\in \mathcal{H}(E),$ $k\in \mathbb{N}\cup \{0\},$ and $\rho\in(1,\rho_0(G)).$ Then,
\begin{equation}\label{usealot}
[G]_k=\frac{1}{2\pi i} \int_{\Gamma_\rho} G(w) s_k(w) dw.
\end{equation}
where $s_k$ is the $k$-th second type function.
\end{lemma}
\begin{proof}[Proof of Lemma \ref{usealot1}]
Let $G\in \mathcal{H}(E),$ $k\in \mathbb{N}\cup \{0\},$ and $\rho\in(1,\rho_0(G)).$ By Cauchy's integral formula and Fubini's theorem, we obtain
$$[G]_k=\langle G,\, p_k \rangle_{\mu}=\int G (z) \overline{p_k(z)} d\mu(z)=\int \frac{1}{2\pi i}\int_{\Gamma_\rho} \frac{G(w)}{w-z} dw \overline{p_k(z)} d\mu(z)$$
$$
=\frac{1}{2\pi i} \int_{\Gamma_\rho} G(w) \int \frac{\overline{p_k(z)}}{w-z} d\mu(z) dw=\frac{1}{2\pi i} \int_{\Gamma_\rho} G(w) s_k(w) dw.
$$
\end{proof}

\subsection{Proof of (a) implies (b)}

 Let $Q_{\textup{\textbf{m}}}^{\textup{\textbf{F}}}$ denote the monic polynomial whose zeros are the system poles of $\textup{\textbf{F}}$ with respect to $\textup{\textbf{m}}$ taking account of their order. The set of distinct zeros of $Q_{\textup{\textbf{m}}}^{\textup{\textbf{F}}}$ is denoted by $\mathcal{P}(\textup{\textbf{F}},\textup{\textbf{m}}).$

We have

\begin{thm}\label{thm1.4} Let $\textup{\textbf{F}}=(F_1,F_2,\ldots,F_d)\in \mathcal{H}(E)^d,$  $\mu\in \textup{\textbf{Reg}}_{1,2}(E),$ and $\textup{\textbf{m}}\in \mathbb{N}^d$ be a fixed multi-index. Suppose that $\textup{\textbf{F}}$ has exactly $|\textup{\textbf{m}}|$ system poles with respect to $\textup{\textbf{m}}$ counting multiplicities. Then, the denominators of the orthogonal Hermite-Pad\'e approximants $Q_{n,|\textup{\textbf{m}}|}^{\mu}$ are uniquely determined for all  sufficiently large $n$ and
\begin{equation}\label{2.5}
\limsup_{n \rightarrow \infty} \|Q_{n,|\textup{\textbf{m}}|}^{\mu}-Q_{\textup{\textbf{m}}}^{\textup{\textbf{F}}}\|^{1/n}=
\max \left\{\frac{|\Phi(\xi)|}{\boldsymbol\rho_{\xi}(\textup{\textbf{F}},\textup{\textbf{m}})} :\xi\in\mathcal{P}(\textup{\textbf{F}},\textup{\textbf{m}})\right\},
\end{equation}
where $\|\cdot\|$ denotes the coefficient norm in the space of polynomials.
Additionally, for each $i=1,\ldots,d,$  $k=0,\ldots,m_i-1,$ and for any compact subset $K$ of $D_{i,k}^{*}(\textup{\textbf{F}},\textup{\textbf{m}})\setminus \mathcal{P}(\textup{\textbf{F}},\textup{\textbf{m}}),$
\begin{equation}\label{approximation}
\limsup_{n \rightarrow \infty} \left\|\frac{P_{n,\textup{\textbf{m}},k,i}^{\mu}}{Q_{n,|\textup{\textbf{m}}|}^{\mu}}-z^{k}F_i\right\|_{K}^{1/n}\leq \frac{\|\Phi\|_K}{\boldsymbol \rho_{i,k}^{*}({\textup{\textbf{F}},\textup{\textbf{m}}})},
\end{equation}
where $\|\cdot \|_{K}$ denotes the sup-norm on $K$ and if $K\subset E,$ then $\|\Phi\|_K$ is replaced by $1.$
\end{thm}

\begin{proof}[Proof of Theorem \ref{thm1.4}]
For each $n\in \mathbb{N},$ let  $q_{n,|\textup{\textbf{m}}|}^{\mu}$ be the polynomial $Q_{n,|\textup{\textbf{m}}|}^{\mu}$ normalized so that
\begin{equation}\label{contradict}
\sum_{k=0}^{|\textup{\textbf{m}}|} |\lambda_{n,k}|=1,\quad \quad q_{n,|\textup{\textbf{m}}|}^{\mu}(z)=\sum_{k=0}^{|\textup{\textbf{m}}|} \lambda_{n,k} z^{k}.
\end{equation} This normalization implies that the polynomials $q_{n,|\textup{\textbf{m}}|}^{\mu}$ are uniformly bounded on each compact subset of $\mathbb{C}.$

Let $\xi$ be a system pole of order $\tau$  of $\textup{\textbf{F}}$ with respect to $\textup{\textbf{m}}.$
We wish to show that
\begin{equation}\label{3.31}
\limsup_{n \rightarrow \infty} |(q_{n,|\textup{\textbf{m}}|}^{\mu})^{(j)}(\xi)|^{1/n}\leq \frac{|\Phi(\xi)|}{\boldsymbol \rho_{\xi,j+1}(\textup{\textbf{F}},\textup{\textbf{m}})}, \quad \quad j=0,1,\ldots,\tau-1.
\end{equation}

First, we consider a polynomial combination $G_1$ of type \eqref{polycom} that is holomorphic on a neighborhood of $\overline{D}_{|\Phi(\xi)|}$ except for a simple pole at $z=\xi$ and verifies that $\rho_{1}(G_1)=\boldsymbol \rho_{\xi,1}(\textup{\textbf{F}},\textup{\textbf{m}})(=\rho_{\xi,1}(\textup{\textbf{F}},\textup{\textbf{m}})).$ Then, we have
$$G_1=\sum_{i=1}^{d} v_{i,1} F_i, \quad \quad \deg v_{i,1}<m_i,\quad i=1,2,\ldots,d.$$
Define $$ H_1(z):=(z-\xi)G_1(z) \quad \quad  \textup{and} \quad \quad a_{n,n}^{(1)}:=[q_{n,|\textup{\textbf{m}}|}^{\mu}G_1]_{n}.$$
By the definition of $Q_{n,|\textup{\textbf{m}}|}^{\mu}$, it is easy to check that $a_{n,n}^{(1)}=0.$
Moreover, using \eqref{usealot}, we have
 $$a_{n,n}^{(1)}=[q_{n,|\textup{\textbf{m}}|}^{\mu}G_1]_n=\frac{1}{2\pi i} \int_{\Gamma_{\rho_1}} q_{n,|\textup{\textbf{m}}|}^{\mu}(z)G_1(z) s_n(z)dz,$$
where $1< \rho_1< |\Phi(\xi)|.$
  Define
 $$\tau_{n,n}^{(1)}:=\frac{1}{2\pi i} \int_{\Gamma_{\rho_2}} q_{n,|\textup{\textbf{m}}|}^{\mu}(z)G_1(z) s_n(z)dz,$$
 where $|\Phi(\xi)|<\rho_2< \boldsymbol \rho_{\xi,1}(\textup{\textbf{F}},\textup{\textbf{m}}).$
The function $q_{n,|\textup{\textbf{m}}|}^{\mu}G_1s_{n}$ is meromorphic on $\overline{D_{\rho_2}}\setminus D_{\rho_1} =\{z\in \mathbb{C}: \rho_1 \leq |\Phi(z)| \leq \rho_2 \}$ and has a pole at $\xi$ of order at most $1.$ Applying Cauchy's residue theorem to the function $q_{n,|\textup{\textbf{m}}|}^{\mu} G_1 s_n$, we have
\begin{align}\label{banana25}
&\frac{1}{2 \pi i}\int_{\Gamma_{\rho_2}} q_{n,|\textup{\textbf{m}}|}^{\mu}(t) G_1(t) s_n(t)dt-\frac{1}{2 \pi i} \int_{\Gamma_{\rho_1}}q_{n,|\textup{\textbf{m}}|}^{\mu}(t) G_1(t) s_n(t) dt \notag\\
&= \textup{res}(q_{n,|\textup{\textbf{m}}|}^{\mu} G_1 s_n,\, \xi).
\end{align}
 The limit formula for the residue of $q_{n,|\textup{\textbf{m}}|}^{\mu}G_1 s_n$ at $\xi$ is
$$
\textup{res}(q_{n,|\textup{\textbf{m}}|}^{\mu}F s_n,\, \xi)= \lim_{z \rightarrow \xi} (z-\xi) q_{n,|\textup{\textbf{m}}|}^{\mu}(z) G_1(z) s_n(z)=H_1(\xi) q_{n,|\textup{\textbf{m}}|}^{\mu}(\xi) s_n(\xi).
$$ We can rewrite \eqref{banana25} as
$$
\tau_{n,n}^{(1)}=\tau_{n,n}^{(1)}-a_{n,n}^{(1)}=H_1(\xi) q_{n,|\textup{\textbf{m}}|}^{\mu}(\xi) s_n(\xi)
$$
(recall that $a_{n,n}^{(1)}=0$) which implies
\begin{equation}\label{rewrite}
q_{n,|\textup{\textbf{m}}|}^{\mu}(\xi) =\frac{\tau_{n,n}^{(1)}}{H_1(\xi)s_n(\xi)}.
\end{equation}

Choose $\delta>0$ so small that
\begin{equation}\label{yuhyggt1}
\rho_2:=\boldsymbol \rho_{\xi,1}(\textup{\textbf{F}},\textup{\textbf{m}})-\delta>|\Phi(\xi)|, \quad |\Phi(\xi)|-\delta>1, \quad  \textup{and} \quad
\frac{|\Phi(\xi)|+\delta}{\rho_2-\delta}<1.
\end{equation}
Using  \eqref{asintlog2},   there exist $n_0\in \mathbb{N}$ and $c_1>0,c_2>0$  such that
\begin{equation}\label{asymps}
  \frac{c_1}{(\rho+\delta)^{n}} \leq \|s_n\|_{\Gamma_{\rho}} \leq \frac{c_2}{(\rho-\delta)^n}, \quad \quad n\geq n_0,
\end{equation}
where $c_1$ and $c_2$  do not depend on $n$ (from now on, $c_3,c_4,\ldots$ denote constants that do not depend on $n$). From \eqref{asymps}, we have
\begin{equation}\label{use1}
|\tau_{n,n}^{(1)}|=\left|\frac{1}{2\pi i} \int_{\Gamma_{\rho_2}} q_{n,|\textup{\textbf{m}}|}^{\mu}(z)G_1(z) s_n(z)dz\right|\leq \frac{c_3}{(\rho_2-\delta)^n}
\end{equation}
and
\begin{equation}\label{use2}
|s_n(\xi)|\geq \frac{c_1}{(|\Phi(\xi)|+\delta)^n}.
\end{equation}
Combining \eqref{use1} and \eqref{use2}, it follows from \eqref{rewrite} that
$$|q_{n,|\textup{\textbf{m}}|}^{\mu}(\xi)|\leq c_4 \left( \frac{|\Phi(\xi)|+\delta}{\rho_2-\delta}\right)^n.$$ which means that
$$\limsup_{n \rightarrow \infty}|q_{n,|\textup{\textbf{m}}|}^{\mu}(\xi)|^{1/n} \leq \frac{|\Phi(\xi)|+\delta}{\rho_2-\delta}.$$
Letting $\delta \rightarrow 0,$ we obtain  $\rho_2\rightarrow {\boldsymbol \rho_{\xi,1}(\textup{\textbf{F}},\textup{\textbf{m}})}$ and
$$\limsup_{n \rightarrow \infty}|q_{n,|\textup{\textbf{m}}|}^{\mu}(\xi)|^{1/n} \leq \frac{|\Phi(\xi)|}{\boldsymbol \rho_{\xi,1}(\textup{\textbf{F}},\textup{\textbf{m}})}.$$

Now we employ induction. Suppose that
\begin{equation}\label{useos}
\limsup_{n \rightarrow \infty} |(q_{n,|\textup{\textbf{m}}|}^{\mu})^{(j)}(\xi)|^{1/n}\leq \frac{|\Phi(\xi)|}{\boldsymbol \rho_{\xi,j+1}(\textup{\textbf{F}},\textup{\textbf{m}})}, \quad \quad j=0,1,\ldots, \ell-2,
\end{equation}
(recall that $\boldsymbol \rho_{\xi,j+1}(\textup{\textbf{F}},\textup{\textbf{m}})=\min_{k=1,\ldots,j+1} \rho_{\xi,k}(\textup{\textbf{F}},\textup{\textbf{m}})$), with $\ell\leq \tau,$ and let us prove that the formula \eqref{useos} holds for $j=\ell-1.$

Consider a polynomial combination $G_\ell$ of type \eqref{polycom} that is holomorphic on a neighborhood of $\overline{D}_{|\Phi(\xi)|}$ except for a pole of order $\ell$ at $z=\xi$ and verifies that $\rho_{\ell}(G_\ell)= \rho_{\xi,\ell}(\textup{\textbf{F}},\textup{\textbf{m}}).$ Then, we have
$$G_\ell=\sum_{i=1}^d v_{i,\ell} F_i, \quad \quad \deg v_{i,\ell}< m_i, \quad i=1,2,\ldots,d.$$
Set $$H_\ell(z):=(x-\xi)^{\ell} G_\ell(z) \quad \quad \textup{and}\quad \quad a_{n,n}^{(\ell)}=[q_{n,|\textup{\textbf{m}}|}^{\mu}G_\ell]_n.$$ By the definition of $Q_{n,|\textup{\textbf{m}}|}^{\mu}$, it is easy to check that
$a_{n,n}^{(\ell)}=0.$
 Using \eqref{usealot}, we have
 $$a_{n,n}^{(\ell)}=[q_{n,|\textup{\textbf{m}}|}^{\mu}G_\ell]_n=\frac{1}{2\pi i} \int_{\Gamma_{\rho_1}} q_{n,|\textup{\textbf{m}}|}^{\mu}(z)G_\ell(z) s_n(z)dz,$$
where $1< \rho_1< |\Phi(\xi)|.$ Define
 $$\tau_{n,n}^{(\ell)}=\frac{1}{2\pi i} \int_{\Gamma_{\rho_2}} q_{n,|\textup{\textbf{m}}|}^{\mu}(z)G_\ell(z) s_n(z)dz,$$
 where $|\Phi(\xi)|<\rho_2<\rho_{\xi,\ell}(\textup{\textbf{F}},\textup{\textbf{m}}).$ The function $q_{n,|\textup{\textbf{m}}|}^{\mu}G_\ell s_{n}$ is meromorphic on $\overline{D_{\rho_2}}\setminus D_{\rho_1} =\{z\in \mathbb{C}: \rho_1 \leq |\Phi(z)| \leq \rho_2 \}$ and has a pole at $\xi$ of order at most $\ell.$ Applying Cauchy's residue theorem to the function $q_{n,|\textup{\textbf{m}}|}^{\mu} G_\ell s_n$, we have
$$
\tau_{n,n}^{(\ell)}-a_{n,n}^{(\ell)}=\frac{1}{2 \pi i}\int_{\Gamma_{\rho_2}} q_{n,|\textup{\textbf{m}}|}^{\mu}(t) G_\ell(t) s_n(t)dt-\frac{1}{2 \pi i} \int_{\Gamma_{\rho_1}}q_{n,|\textup{\textbf{m}}|}^{\mu}(t) G_\ell(t) s_n(t) dt
$$
\begin{equation}\label{reduce}
= \textup{res}(q_{n,|\textup{\textbf{m}}|}^{\mu} G_\ell s_n,\, \xi).
\end{equation}
The limit formula for the residue of $q_{n,|\textup{\textbf{m}}|}^{\mu}G_\ell s_n$ at $\xi$ is
$$
\textup{res}(q_{n,|\textup{\textbf{m}}|}^{\mu}G_\ell s_n,\, \xi)=\frac{1}{(\ell-1)!} \lim_{z \rightarrow \xi} ((z-\xi)^{\ell} G_\ell(z)  s_n(z) q_{n,|\textup{\textbf{m}}|}^{\mu}(z))^{(\ell-1)}
$$
\begin{equation}\label{reduce111111}
=\frac{1}{(\ell-1)!} \sum_{t=0}^{\ell-1} {\ell-1 \choose t }  (H_\ell s_n)^{(\ell-1-t)}(\xi) (q_{n,|\textup{\textbf{m}}|}^{\mu})^{(t)}(\xi),
\end{equation}
where the last equality follows from Leibniz's formula.
Since $a_{n,n}^{(\ell)}=0,$ the equation \eqref{reduce} becomes
$$(\ell-1)!\tau_{n,n}^{(\ell)}= \sum_{t=0}^{\ell-2} {\ell-1 \choose t } (H_\ell s_n)^{(\ell-1-t)}(\xi) (q_{n,|\textup{\textbf{m}}|}^{\mu})^{(t)}(\xi)+H_\ell(\xi)s_n(\xi) (q_{n,|\textup{\textbf{m}}|}^{\mu})^{(\ell-1)}(\xi),
$$
which implies that
\begin{equation}\label{use678}
(q_{n,|\textup{\textbf{m}}|}^{\mu})^{(\ell-1)}(\xi)=\frac{(\ell-1)!\tau_{n,n}^{(\ell)}}{H_\ell(\xi)s_n(\xi) }-\sum_{t=0}^{\ell-2} {\ell-1 \choose t } \frac{(H_\ell s_n)^{(\ell-1-t)}(\xi) (q_{n,|\textup{\textbf{m}}|}^{\mu})^{(t)}(\xi)}{H_\ell(\xi)s_n(\xi) }.
\end{equation}

Choose $\delta>0$ sufficiently small so that
\begin{equation}\label{yuhyggt1}
\rho_2:=\rho_{\xi,\ell}(\textup{\textbf{F}},\textup{\textbf{m}})-\delta>|\Phi(\xi)|, \quad |\Phi(\xi)|-\delta>1, \quad  \textup{and} \quad
\frac{|\Phi(\xi)|+\delta}{\rho_2-\delta}<1.
\end{equation}
Using \eqref{asymps}, we have \begin{equation}\label{use1y}
|\tau_{n,n}^{(\ell)}|=\left|\frac{1}{2\pi i} \int_{\Gamma_{\rho_2}} q_{n,|\textup{\textbf{m}}|}^{\mu}(z)G_\ell(z) s_n(z)dz\right|\leq \frac{c_5}{(\rho_2-\delta)^n},
\end{equation}
\begin{equation}\label{use2y}
|s_n(\xi)|\geq \frac{c_1}{(|\Phi(\xi)|+\delta)^n},
\end{equation}
and for all $t=0,1,\ldots,\ell-2,$
\begin{equation}\label{use3y}
|(H_\ell s_n)^{(\ell-1-t)}(\xi)|=\left|\frac{(\ell-1-t)!}{2\pi i} \int_{|z-\xi|=\varepsilon}\frac{H_\ell(z) s_n(z)}{(z-\xi)^{\ell-t}}dz\right|\leq \frac{c_6}{(|\Phi(\xi)|-\delta)^n},
\end{equation}
where $\{z\in \mathbb{C}: |z-\xi|=\varepsilon\}\subset \{z\in \mathbb{C}: |\Phi(z)|>|\Phi(\xi)|-\delta\}.$
Moreover, by \eqref{useos}, we have for all $j=0,1,\ldots, \ell-2,$
\begin{equation}\label{use4y}
|(q_{n,|\textup{\textbf{m}}|}^{\mu})^{(j)}(\xi)|\leq c_7  \left(\frac{|\Phi(\xi)|}{\boldsymbol \rho_{\xi,j+1}(\textup{\textbf{F}},\textup{\textbf{m}})}\right)^n \leq c_7  \left(\frac{|\Phi(\xi)|}{\boldsymbol \rho_{\xi,\ell-1}(\textup{\textbf{F}},\textup{\textbf{m}})}\right)^n.
\end{equation}
Combining  \eqref{use1y}, \eqref{use2y}, \eqref{use3y} and \eqref{use4y}, it follows from \eqref{use678} that
$$\left|(q_{n,|\textup{\textbf{m}}|}^{\mu})^{(\ell-1)}(\xi)\right|=\left|\frac{(\ell-1)!\tau_{n,n}^{(\ell)}}{H_\ell(\xi)s_n(\xi) }-\sum_{t=0}^{\ell-2} {\ell-1 \choose t } \frac{(H_\ell s_n)^{(\ell-1-t)}(\xi) (q_{n,|\textup{\textbf{m}}|}^{\mu})^{(t)}(\xi)}{H_\ell(\xi)s_n(\xi) }\right|$$
$$\leq c_8\left( \frac{|\Phi(\xi)|+\delta}{\rho_2-\delta}\right)^n+c_9\left( \frac{|\Phi(\xi)|+\delta}{|\Phi(\xi)|-\delta}\right)^n \left(\frac{|\Phi(\xi)|}{\boldsymbol \rho_{\xi,\ell-1}(\textup{\textbf{F}},\textup{\textbf{m}})}\right)^n,$$
which implies that
\begin{equation}\label{have}
\limsup_{n \rightarrow \infty} \left|(q_{n,|\textup{\textbf{m}}|}^{\mu})^{(\ell-1)}(\xi)\right|^{1/n}\leq  \max \left\{ \frac{|\Phi(\xi)|+\delta}{\rho_2-\delta}, \left( \frac{|\Phi(\xi)|+\delta}{|\Phi(\xi)|-\delta}\right)\left(\frac{|\Phi(\xi)|}{\boldsymbol \rho_{\xi,\ell-1}(\textup{\textbf{F}},\textup{\textbf{m}})}\right) \right\}.
\end{equation}
Letting $\delta \rightarrow 0,$ we have $\rho_2\rightarrow  \rho_{\xi,\ell}(\textup{\textbf{F}},\textup{\textbf{m}})$ and from \eqref{have}, we obtain
$$\limsup_{n \rightarrow \infty} \left|(q_{n,|\textup{\textbf{m}}|}^{\mu})^{(\ell-1)}(\xi)\right|^{1/n}\leq  \max \left\{\frac{|\Phi(\xi)|}{\rho_{\xi,\ell}(\textup{\textbf{F}},\textup{\textbf{m}})} , \frac{|\Phi(\xi)|}{ \boldsymbol \rho_{\xi,\ell-1}(\textup{\textbf{F}},\textup{\textbf{m}})} \right\}
$$
$$\leq \frac{|\Phi(\xi)|}{\boldsymbol \rho_{\xi,\ell}(\textup{\textbf{F}},\textup{\textbf{m}})}.$$
This completes the induction proof.

Let $\xi_1,\ldots,\xi_w$ be the distinct system poles of $\textup{\textbf{F}}$ with respect to $\textup{\textbf{m}}$, and let $\tau_j$ be the order of $\xi_j$ as a
system pole, $j=1,\ldots,w.$ By assumption, $\tau_1+\ldots+\tau_w =|\textup{\textbf{m}}|.$ We have proved that, for $j =1,\ldots,w$ and $t =0,1,\ldots,\tau_j-1,$
\begin{equation}\label{3.32}
\limsup_{n \rightarrow \infty} |(q_{n,|\textup{\textbf{m}}|}^{\mu})^{(t)}(\xi_j)|^{1/n}\leq \frac{|\Phi(\xi_j)|}{\boldsymbol \rho_{\xi_j,t+1}(\textup{\textbf{F}},\textup{\textbf{m}})}\leq \frac{|\Phi(\xi_j)|}{\boldsymbol \rho_{\xi_j} (\textup{\textbf{F}},\textup{\textbf{m}})}.
\end{equation}
Using Hermite interpolation, it is easy to construct a basis  $\{e_{j,t}\}_{j=1,2,\ldots,w,\, t=0,1,\ldots,\tau_j-1}$ in the space of polynomials of degree at most $|\textup{\textbf{m}}|-1$ satisfying
$$e_{j,t}^{(k)}(\xi_i)=\delta_{i,j}\delta_{k,t},\quad \quad 1\leq i \leq w, \quad \quad 0\leq k \leq \tau_i-1.$$ Then,
\begin{equation}\label{3.33}
q_{n,|\textup{\textbf{m}}|}^{\mu}(z)=\sum_{j=1}^{w} \sum_{t=0}^{\tau_j-1} (q_{n,|\textup{\textbf{m}}|}^{\mu})^{(t)}(\xi_j)e_{j,t}(z)+C_n Q_{\textup{\textbf{m}}}^{\textup{\textbf{F}}}(z),
\end{equation}
Using \eqref{3.32} and \eqref{3.33}, we have
\begin{equation}\label{need}
\limsup_{n \rightarrow \infty} \|q_{n,|\textup{\textbf{m}}|}^{\mu}-C_{n} Q_{|\textup{\textbf{m}}|}^{{\textup{\textbf{F}}}}\|^{1/n}\leq
\max \left\{\frac{|\Phi(\xi)|}{\boldsymbol\rho_{\xi}(\textup{\textbf{F}},\textup{\textbf{m}})} :\xi\in\mathcal{P}(\textup{\textbf{F}},\textup{\textbf{m}})\right\}.
\end{equation}
Now, necessarily we have
\begin{equation}
\label{need2}
\liminf_{n \rightarrow \infty} |C_n|>0,
\end{equation}
since if there exists a subsequence $\Lambda\subset \mathbb{N}$ such that $\lim_{n\in \Lambda} C_n=0,$ then from \eqref{need}, we have $\lim_{n \in \Lambda} \|q_{n,|\textup{\textbf{m}}|}^{\mu}\|=0,$ contradicting \eqref{contradict}.

As $q_{n,|\textup{\textbf{m}}|}^{\mu}=C_n Q_{n,|\textup{\textbf{m}}|}^{\mu},$ we have proved
\begin{equation}\label{usenewone}
\limsup_{n \rightarrow \infty} \|Q_{n,|\textup{\textbf{m}}|}^{\mu}-Q_{|\textup{\textbf{m}}|}^{{\textup{\textbf{F}}}}\|^{1/n}\leq \max \left\{\frac{|\Phi(\xi)|}{\boldsymbol\rho_{\xi}(\textup{\textbf{F}},\textup{\textbf{m}})} :\xi\in\mathcal{P}(\textup{\textbf{F}},\textup{\textbf{m}})\right\}.
\end{equation}
In particular, for $n\geq n_0,$ $\deg Q_{n,|\textup{\textbf{m}}|}^{\mu}=|\textup{\textbf{m}}|.$ The difference of any two distinct monic polynomials satisfying Definition \ref{simu} with the same degree produces  a new solution of degree strictly less than $|\textup{\textbf{m}}|$, but we have proved that any solution must have degree $|\textup{\textbf{m}}|$ for all sufficiently large $n.$ Hence, the polynomial $Q_{n,|\textup{\textbf{m}}|}^{\mu}$ is uniquely determined for all sufficiently large $n,$

Now, we prove the equality in \eqref{2.5}. To the contrary, suppose that
\begin{equation}\label{less}
\limsup_{n \rightarrow \infty} \|Q_{n,|\textup{\textbf{m}}|}^{\mu}-Q_{\textup{\textbf{m}}}^{\textup{\textbf{F}}}\|^{1/n}=\theta<
\max \left\{\frac{|\Phi(\xi)|}{\rho_{\xi}(\textup{\textbf{F}},\textup{\textbf{m}})} :\xi\in\mathcal{P}(\textup{\textbf{F}},\textup{\textbf{m}})\right\}.
\end{equation}
Let $\zeta$ be a system pole of $\textup{\textbf{F}}$ such that
$$\frac{|\Phi(\zeta)|}{\boldsymbol \rho_{\zeta}(\textup{\textbf{F}},\textup{\textbf{m}})}=\max \left\{\frac{|\Phi(\xi)|}{\rho_{\xi}(\textup{\textbf{F}},\textup{\textbf{m}})} :\xi\in\mathcal{P}(\textup{\textbf{F}},\textup{\textbf{m}})\right\}.$$ Clearly, the inequality \eqref{less} implies that $\boldsymbol \rho_{\zeta}(\textup{\textbf{F}},\textup{\textbf{m}})<\infty.$

Choose a polynomial combination
\begin{equation}\label{uhbjikml}
G=\sum_{i=1}^d v_i F_i,\quad \quad \deg v_i < m_i,\quad i=1,2,\ldots,d,
\end{equation}
that is holomorphic on a neighborhood of $\overline{D}_{|\Phi(\zeta)|}$ except for a pole of order $s$ at $z=\zeta$ with $\rho_s(G)=\boldsymbol \rho_{\zeta}(\textup{\textbf{F}},\textup{\textbf{m}}).$ On the boundary of $D_{\rho_s}(G),$ the function $G$ must have a singularity which is not a system pole.  In fact, if all the singularities were of this type, then we could find a different polynomial combination $G_1$ of type \eqref{uhbjikml} for which $\rho_s(G_1)>\rho_s(G)=\boldsymbol \rho_{\zeta}(\textup{\textbf{F}},\textup{\textbf{m}}),$ which contradicts the definition of $\boldsymbol \rho_{\zeta}(\textup{\textbf{F}},\textup{\textbf{m}})$. Therefore, by Lemma \ref{expan},
$$\limsup_{n \rightarrow \infty} |[Q_{\textup{\textbf{m}}}^{\textup{\textbf{F}}}G]_n|^{1/n}=\frac{1}{\boldsymbol \rho_{\zeta}(\textup{\textbf{F}},\textup{\textbf{m}})}.$$

Choose $1<\rho<|\Phi(\zeta)|$ and $\varepsilon>0.$
Then, by the definition of $Q_{n,|\textup{\textbf{m}}|}^{\mu},$ \eqref{usealot}, and \eqref{less},
$$\frac{1}{\boldsymbol \rho_{\zeta}(\textup{\textbf{F}},\textup{\textbf{m}})}=\limsup_{n \rightarrow \infty} |[Q_{\textup{\textbf{m}}}^{\textup{\textbf{F}}}G]_n|^{1/n}=\limsup_{n \rightarrow \infty} |[Q_{\textup{\textbf{m}}}^{\textup{\textbf{F}}}G-Q_{n,|\textup{\textbf{m}}|}^{\mu} G]_n|^{1/n}$$
$$=\limsup_{n \rightarrow \infty}\left|\frac{1}{2\pi i} \int_{\Gamma_{\rho}} (Q_{\textup{\textbf{m}}}^{\textup{\textbf{F}}}-Q_{n,|\textup{\textbf{m}}|}^{\mu})(z) G(z) s_n(z) dz \right|^{1/n}\leq \frac{\theta}{\rho-\varepsilon}.$$
Letting $\varepsilon\rightarrow 0$ and $\rho\rightarrow |\Phi(\zeta)|$ in the above inequality, we have
$$\frac{1}{\boldsymbol \rho_{\zeta}(\textup{\textbf{F}},\textup{\textbf{m}})}\leq \frac{\theta}{|\Phi(\zeta)|}< \frac{|\Phi(\zeta)|/\boldsymbol \rho_{\zeta}(\textup{\textbf{F}},\textup{\textbf{m}})}{|\Phi(\zeta)|}=\frac{1}{\boldsymbol \rho_{\zeta}(\textup{\textbf{F}},\textup{\textbf{m}})},$$ which is impossible. This proves the equality \eqref{2.5}.

Let us prove the inequality \eqref{approximation}.
Combining \eqref{3.32} and \eqref{need2}, it follows that for the system poles $\xi_1,\ldots,\xi_w$ of $\textup{\textbf{F}}$, if $\tau_j$ is the order of $\xi_j,$ then
\begin{equation}\label{3.31111}
\limsup_{n \rightarrow \infty} |(Q_{n,|\textup{\textbf{m}}|}^{\mu})^{(u)}(\xi_j)|^{1/n}\leq \frac{|\Phi(\xi_j)|}{\boldsymbol \rho_{\xi_j,u+1}(\textup{\textbf{F}},\textup{\textbf{m}})}, \quad \quad u=0,1,\ldots,\tau_j-1.
\end{equation} Let $i\in \{1,\ldots,d\}$ and $k\in \{0,1,\ldots,m_i-1\}$ be fixed and let $\tilde{\xi}_1,\ldots,\tilde{\xi}_N$ be the poles of $z^{k}F_i$ in $D_{i,k}(\textup{\textbf{F}},\textup{\textbf{m}}).$ For each $j=1,\ldots,N,$ let $\hat{\tau}_j$ be the order of $\tilde{\xi}_j$ as a pole of $z^{k}F_i$ and $\tilde{\tau}_j$ its order as a system pole. Recall that by assumption, $\hat{\tau}_j \leq\tilde{\tau}_j.$  Define
$$a_{\ell,n}^{(i,k)}:=[Q_{n,|\textup{\textbf{m}}|}^{\mu} z^{k}F_{i}]_{\ell}=\frac{1}{2\pi i} \int_{\Gamma_{\rho_1}} Q_{n,|\textup{\textbf{m}}|}^{\mu}(z) z^{k}F_{i}(z) s_\ell(z) dz,$$
where $1< \rho_1< \rho_{0}(z^{k}F_i)$ and
$$\tau_{\ell,n}^{(i,k)}:=[Q_{n,|\textup{\textbf{m}}|}^{\mu} z^{k}F_{i}]_{\ell}=\frac{1}{2\pi i} \int_{\Gamma_{\rho_2}} Q_{n,|\textup{\textbf{m}}|}^{\mu}(z) z^{k}F_{i}(z) s_\ell(z) dz,$$
where $1< \rho_2< \boldsymbol \rho_{i,k}^{*}(\textup{\textbf{F}},\textup{\textbf{m}}).$
Arguing as in \eqref{reduce} and \eqref{reduce111111}, we have
$$
\tau_{\ell,n}^{(i,k)}-a_{\ell,n}^{(i,k)}=\sum_{j=1}^N\textup{res}(Q_{n,|\textup{\textbf{m}}|}^{\mu} z^{k} F_i s_\ell,\, \tilde{\xi}_j)
$$
\begin{equation}\label{qwertyuiopsdfghjk}
=\sum_{j=1}^N \frac{1}{(\hat{\tau}_j-1)!} \sum_{u=0}^{\hat{\tau}_j-1} {\hat{\tau}_j-1 \choose u }  ((z-\tilde{\xi}_j)^{\hat{\tau}_j}z^{k}F_i s_\ell)^{(\hat{\tau}_j-1-u)}(\tilde{\xi}_j) (Q_{n,|\textup{\textbf{m}}|}^{\mu})^{(u)}(\tilde{\xi}_j).
\end{equation}
Notice that $(z-\tilde{\xi}_j)^{\hat{\tau}_j}z^{k}F_i$ is holomorphic at $\tilde{\xi}_j.$  Let $\delta>0$ be such that $\rho_2-\delta>1$ and $|\Phi(\tilde{\xi}_j)|-\delta>1$ (later on we will impose another condition on the size of $\delta$).
By computations similar to \eqref{use1y} and \eqref{use3y}, we have
\begin{equation}\label{useuseuse}
|\tau_{\ell,n}^{(i,k)}|\leq \frac{c_{10}}{(\rho_2-\delta)^\ell}\quad \quad \textup{and}\quad \quad |((z-\tilde{\xi}_j)^{\hat{\tau}_j}z^{k}F_i s_\ell)^{(\hat{\tau}_j-1-u)}(\tilde{\xi}_j)|\leq \frac{c_{11}}{(|\Phi(\tilde{\xi}_j)|-\delta)^\ell},
\end{equation}
respectively.
By \eqref{3.31111} and \eqref{useuseuse}, it follows from \eqref{qwertyuiopsdfghjk} that
$$
|a_{\ell,n}^{(i,k)}|=|\tau_{\ell,n}^{(i,k)}|+\sum_{j=1}^N \sum_{u=0}^{\hat{\tau}_j-1} \frac{1}{(\hat{\tau}_j-1)!}  {\hat{\tau}_j-1 \choose u }  \left|((z-\tilde{\xi}_j)^{\hat{\tau}_j}z^{k}F_i s_\ell)^{(\hat{\tau}_j-1-u)}(\tilde{\xi}_j)\right| \left| (Q_{n,|\textup{\textbf{m}}|}^{\mu})^{(u)}(\tilde{\xi}_j)\right|.
$$
$$
\leq \frac{c_{10}}{(\rho_2-\delta)^\ell}+ c_{12}  \sum_{j=1}^N \frac{ |\Phi(\tilde{\xi}_j)|^n}{(\boldsymbol \rho_{\tilde{\xi}_j,\hat{\tau}_j}(\textup{\textbf{F}},\textup{\textbf{m}}))^n(|\Phi(\tilde{\xi}_j)|-\delta)^\ell}
$$
\begin{equation}\label{approxaaa}
\leq \frac{c_{10}}{(\rho_2-\delta)^\ell}+ \frac{c_{12}}{( \boldsymbol \rho_{i,k}^{*}(\textup{\textbf{F}},\textup{\textbf{m}}))^n} \sum_{j=1}^N \frac{ |\Phi(\tilde{\xi}_j)|^n}{(|\Phi(\tilde{\xi}_j)|-\delta)^\ell}.
\end{equation}
By the definition of orthogonal Hermite-Pad\'e approximants,
$$Q_{n,|\textup{\textbf{m}}|}^{\mu} z^{k} F_i-P_{n,\textup{\textbf{m}},k,i}^{\mu}=\sum_{\ell=n+1}^{\infty} a_{\ell,n}^{(i,k)}p_{\ell}.$$
Multiplying the above equality by $\omega(z):=\prod_{j=1}^N(z-\tilde{\xi}_j)^{\hat{\tau}_j}$ and expanding the result in terms of the Fourier series corresponding to the orthonormal system $\{p_{\nu}\}_{\nu=0}^{\infty}$, we obtain
$$\omega Q_{n,|\textup{\textbf{m}}|}^{\mu} z^{k}F_i-\omega P_{n,\textup{\textbf{m}},k,i}^{\mu}=\sum_{\ell=n+1}^{\infty} a_{\ell,n}^{(i,k)} \omega p_{\ell}$$
\begin{equation}\label{mainmain3}
=\sum_{\nu=0}^{\infty} b_{\nu,n}^{(i,k)} p_{\nu}=\sum_{\nu=0}^{n+|\textup{\textbf{m}}|} b_{\nu,n}^{(i,k)} p_{\nu}+\sum_{\nu=n+|\textup{\textbf{m}}|+1}^{\infty} b_{\nu,n}^{(i,k)} p_{\nu}.
\end{equation}

Let $K$ be a compact subset of $D_{i,k}^{*}{({\textup{\textbf{F}},\textup{\textbf{m}}})}\setminus \mathcal{P}(\textup{\textbf{F}},\textup{\textbf{m}})$ and set
\begin{equation}\label{tgbrdcvhjkl}
 \sigma:=\max\{\|\Phi\|_K,1\}
 \end{equation}
($\sigma=1$ when $K \subset E$). Choose $\delta>0$ so small that
\begin{equation}\label{yuhyggt1}
\rho_2:=\boldsymbol \rho_{i,k}^{*}({\textup{\textbf{F}},\textup{\textbf{m}}})-\delta,\quad \quad \boldsymbol \rho_{i,k}^{*}({\textup{\textbf{F}},\textup{\textbf{m}}})-2\delta>1, \quad \quad  \textup{and} \quad  \quad
\frac{\sigma+\delta}{\rho_2-\delta}<1.
\end{equation}

Let us estimate $\sum_{\nu=n+|\textup{\textbf{m}}|+1}^{\infty} |b_{\nu,n}^{(i,k)}| |p_{\nu}|$ on $\overline{D}_{\sigma}.$ For $\nu \geq n+ |\textup{\textbf{m}}|+1,$
$$b_{\nu,n}^{(i,k)}=[\omega Q_{n,|\textup{\textbf{m}}|}^{\mu} z^{k}F_i-\omega P_{n,\textup{\textbf{m}},k,i}^{\mu}]_{\nu}=[\omega Q_{n,|\textup{\textbf{m}}|}^{\mu} z^{k}F_i]_{\nu}$$
$$=\frac{1}{2\pi i}\int_{\Gamma_{\rho_2}} z^{k}\omega(z) Q_{n,|\textup{\textbf{m}}|}^{\mu}(z) F_i(z) s_{\nu}(z) dz,$$
where  $1< \rho_2<  \boldsymbol \rho^{*}_{i,k}(\textup{\textbf{F}},\textup{\textbf{m}}).$ By a computation similar to \eqref{use1y} or \eqref{use3y}, we obtain
\begin{equation}\label{secondtypeasym}
|b_{\nu,n}^{(i,k)}|\leq  \frac{c_{13}}{ (\rho_2-\delta)^\nu}.
\end{equation}
Moreover, by \eqref{asintlog},
\begin{equation}\label{polynomialasym}
\|p_\nu\|_{\overline{D}_{\sigma}}\leq c_{14} (\sigma+\delta)^{\nu}, \quad \quad \nu\geq 0.
\end{equation}
Combining \eqref{secondtypeasym} and \eqref{polynomialasym}, we have for $z\in \overline{D}_{\sigma},$
$$\sum_{\nu=n+|\textup{\textbf{m}}|+1}^{\infty} |b_{\nu,n}^{(i,k)}| |p_{\nu}(z)|\leq c_{15} \sum_{\nu=n+|\textup{\textbf{m}}|+1}^{\infty}\left( \frac{\sigma+\delta}{\rho_2-\delta}\right)^{\nu}=c_{16} \left( \frac{\sigma+\delta}{\rho_2-\delta}\right)^{n},$$ which implies that
$$\limsup_{n \rightarrow \infty}\left\|\sum_{\nu=n+|\textup{\textbf{m}}|+1}^{\infty} |b_{\nu,n}^{(i,k)}| |p_{\nu}|\right\|_{\overline{D}_{\sigma}}^{1/n}\leq  \frac{\sigma+\delta}{\rho_2-\delta}.$$
Letting $\delta \rightarrow 0,$ we have $\rho_2\rightarrow \boldsymbol \rho_{i,k}^{*}(\textup{\textbf{F}},\textup{\textbf{m}})$ and
\begin{equation}\label{estimate1}
\limsup_{n \rightarrow \infty}\left\|\sum_{\nu=n+|\textup{\textbf{m}}|+1}^{\infty} |b_{\nu,n}^{(i,k)}| |p_{\nu}|\right\|_{\overline{D}_{\sigma}}^{1/n}\leq  \frac{\sigma}{\boldsymbol \rho_{i,k}^{*}(\textup{\textbf{F}},\textup{\textbf{m}})}.
\end{equation}

Now, we want to estimate $\sum_{\nu=0}^{n+|\textup{\textbf{m}}|} |b_{\nu,n}^{(i,k)}| |p_{\nu}|$ on $\overline{D}_{\sigma}.$ Notice that
$$b_{\nu,n}^{(i,k)}=\sum_{\ell=n+1}^{\infty} a_{\ell,n}^{(i,k)}\langle \omega  p_{\ell},\, p_{\nu} \rangle_{\mu}.$$
By the Cauchy-Schwarz inequality and the orthonormality of $p_\nu$, we have
\begin{equation}\label{banana100hougeqs}
|\langle \omega  p_{\ell},\, p_{\nu} \rangle_{\mu}|^2 \leq \langle \omega p_\ell, \, \omega p_\ell \rangle_{\mu} \langle p_\nu, \,p_\nu  \rangle_{\mu} \leq \max_{z\in E}|\omega(z)|^2 = c_{17},
\end{equation}
for all $\ell,\nu=0,1,2,\ldots.$ By \eqref{approxaaa}, we have
$$
|b_{\nu,n}^{(i,k)}|\leq \sum_{\ell=n+1}^{\infty} |a_{\ell,n}^{(i,k)}||\langle \omega  p_{\ell},\, p_{\nu} \rangle_{\mu}|
$$
\begin{equation}\label{banana100hougeqs1}
\leq  \frac{c_{18}}{(\rho_2-\delta)^n}+ \frac{c_{19}}{(\boldsymbol \rho_{i,k}^{*}(\textup{\textbf{F}},\textup{\textbf{m}}))^n}\sum_{j=1}^N \frac{ |\Phi(\tilde{\xi}_j)|^n}{(|\Phi(\tilde{\xi}_j)|-\delta)^n}
\end{equation}
Combining  \eqref{polynomialasym} and \eqref{banana100hougeqs1}, we have for $z\in \overline{D}_{\sigma},$
$$\sum_{\nu=0}^{n+|\textup{\textbf{m}}|} |b_{\nu,n}^{(i,k)}| |p_{\nu}(z)| \leq
$$
$$ c_{20}(n+|\textup{\textbf{m}}|+1) (\sigma+\delta)^{n+|\textup{\textbf{m}}|} \left(  \frac{1}{(\rho_2-\delta)^n}+ \frac{1}{(\boldsymbol \rho_{i,k}^{*}(\textup{\textbf{F}},\textup{\textbf{m}}))^n}\sum_{j=1}^N \frac{ |\Phi(\tilde{\xi}_j)|^n}{(|\Phi(\tilde{\xi}_j)|-\delta)^n} \right),$$
which implies that
$$\limsup_{n \rightarrow \infty} \left\|\sum_{\nu=0}^{n+|\textup{\textbf{m}}|} |b_{\nu,n}^{(i,k)}| |p_{\nu}| \right\|_{\overline{D}_{\sigma}}^{1/n}\leq \max \left\{ \frac{\sigma+\delta}{\rho_2-\delta}, \left(\frac{\sigma+\delta}{\boldsymbol \rho_{i,k}^{*}(\textup{\textbf{F}},\textup{\textbf{m}})} \right)\max_{j=1,\ldots,N} \left(\frac{|\Phi(\tilde{\xi}_j)|}{|\Phi(\tilde{\xi}_j)|-\delta}\right) \right\}.$$
Letting $\delta\rightarrow 0,$ we have $\rho_2 \rightarrow \boldsymbol \rho_{i,k}^{*}(\textup{\textbf{F}},\textup{\textbf{m}})$ and we obtain
\begin{equation}\label{mainmain1}
\limsup_{n \rightarrow \infty} \left\|\sum_{\nu=0}^{n+|\textup{\textbf{m}}|} |b_{\nu,n}^{(i,k)}| |p_{\nu}| \right\|_{\overline{D}_{\sigma}}^{1/n}\leq \frac{\sigma}{ \boldsymbol \rho_{i,k}^{*}(\textup{\textbf{F}},\textup{\textbf{m}})}.
\end{equation}

Using \eqref{usenewone}, \eqref{estimate1}, and \eqref{mainmain1}, it follows from \eqref{mainmain3} that
$$\limsup_{n \rightarrow \infty} \left\| z^{k}F_i -\frac{P_{n,\textup{\textbf{m}},k,i}^{\mu}}{Q_{n,|\textup{\textbf{m}}|}^{\mu}}\right\|_{K}^{1/n}\leq \limsup_{n \rightarrow \infty}\left\| z^{k}F_i- \frac{P_{n,\textup{\textbf{m}},k,i}^{\mu}}{Q_{n,|\textup{\textbf{m}}|}^{\mu} }\right\|_{\overline{D}_{\sigma}}^{1/n}\leq \frac{\sigma}{ \boldsymbol \rho_{i,k}^{*}(\textup{\textbf{F}},\textup{\textbf{m}})}.$$ This completes the proof.
\end{proof}

\section{The inverse statement}

First of all, let us give some inverse type results for incomplete orthogonal Pad\'e approximants.

\subsection{Incomplete orthogonal Pad\'e approximants}

Let us introduce the definition of incomplete orthogonal Pad\'e approximants.

\begin{definition}\label{incomsimu}\textup{ Let $F\in \mathcal{H}(E)$  and $\mu\in \mathcal{M}(E).$ Fix $m \geq m^{*}\geq 1$ and $n \in \mathbb{N}.$  Then, there exist polynomials $Q_{n,m,m^{*}}^{\mu}$ and $P_{n,m,m^{*},k}^{\mu},$ $k=0,1,\ldots,m^{*}-1,$ such that
$$ \deg(P_{n,m,m^{*},k}^{\mu})\leq n-1 \quad \quad \deg(Q_{n,m,m^{*}}^{\mu})\leq m, \quad Q_{n,m,m^{*}}^{\mu}\not\equiv 0,$$
$$\langle Q_{n,m,m^{*}}^{\mu} z^{k} F-P_{n,m,m^{*},k}^{\mu},\, p_{j} \rangle_{\mu}=0,  \quad \quad k=0,1,\ldots,m^{*}-1, \quad \quad j=0,1,\ldots,n.$$
The rational function
$R_{n,m,m^{*}}^{\mu}:=P_{n,m,m^{*},0}^{\mu}/Q_{n,m,m^{*}}^{\mu}$
is called  an \emph{$(n,m,m^{*})$ incomplete orthogonal Pad\'e approximant of $F$ with respect to $\mu$}.}
\end{definition}

Clearly, $$[z^{k}  Q_{n,m,m^{*}}^{\mu}  F]_n=\langle z^{k}  Q_{n,m,m^{*}}^{\mu}  F,\, p_{n} \rangle_{\mu}=0, \quad \quad k=0,1,\ldots,m^{*}-1.$$
In general, $Q_{n,m,m^{*}}^{\mu}$ is not uniquely determined. For each $m\geq  m^{*}\geq 1$ and $n \geq 0$, we choose one candidate of $Q_{n,m,m^{*}}^{\mu}$. Since $Q_{n,m,m^{*}}^{\mu}\not\equiv 0,$ we normalize it to have leading coefficient equal to $1.$ We call $Q_{n,m,m^{*}}^{\mu}$ the denominator of the corresponding $(n,m,m^{*})$ incomplete orthogonal Pad\'e approximant of $F$ with respect to $\mu$. Notice that for each $i=1,\ldots,d,$ $Q_{n,|\textup{\textbf{m}}|}^{\mu}$ is a denominator of an $(n,|\textup{\textbf{m}}|,m_i)$ incomplete orthogonal Pad\'e approximant of $F_i$ with respect to $\mu.$

In this section, we are interested in studying  the relation between the convergence of $Q_{n,m,m^{*}}^{\mu}$ and the analytic properties of $F.$

\begin{lemma}\label{lemma4} Let $F\in \mathcal{H}(E)$ and $\mu\in \textup{\textbf{Reg}}_{1,2}(E).$  Fix $m\geq m^{*}\geq 1.$ Suppose that there exists a polynomial $Q_{m}$ of degree $m$ such that
\begin{equation}\label{3.3332}
\limsup_{n \rightarrow \infty} \|Q_{n,m,m^{*}}^{\mu}-Q_{m}\|^{1/n}=\theta<1.
\end{equation}
Then, $\rho_{0}(Q_{m} F)\geq \rho_{m^{*}}(F).$
\end{lemma}
\begin{proof}[Proof of Lemma \ref{lemma4}]   Let $\xi_1,\ldots,\xi_w$ be  the distinct poles of $F$ in $D_{\rho_{m^{*}}(F)}$ and $\tau_1,\ldots,\tau_{w}$ be their orders, respectively. Consequently, $\sum_{j=1}^{w} \tau_j=\tilde{m}\leq m^{*}.$ Put
$$q_{m^{*}}(z):=\prod_{j=1}^{w}(z-\xi_j)^{\tau_j}.$$
Define
$$G_{j,t}(z):=\frac{q_{m^{*}}(z)F(z)}{(z-\xi_j)^{t}},\quad \quad j=1,\ldots,w,\quad \quad t=1,\ldots,\tau_j-1.$$
Clearly, $G_{j,t}$ is holomorphic on a neighborhood of $\overline{D}_{|\Phi(\xi_j)|}$ except for a pole of order $t$ at $z=\xi_{j}$ with $\rho_{t}(G_{j,t})=\rho_{m^{*}}(F).$ Moreover, since $\deg (q_{m^{*}}/(z-\xi_j)^{t})<m^{*}$ for all $j=1,\ldots,w,$ and $t=1,\ldots,\tau_j-1,$
by the definition of $Q_{n,m,m^{*}}^{\mu},$ it is easy to check that $[Q_{n,m,m^{*}}^{\mu}G_{j,t}]_n=0.$ Arguing as in the proof of \eqref{3.32}, we can prove that  for $j =1,\ldots,w$ and $t =0,1,\ldots,\tau_j-1,$
\begin{equation}\label{3.32new}
\limsup_{n \rightarrow \infty} |(Q_{n,m,m^{*}}^{\mu})^{(t)}(\xi_j)|^{1/n}\leq \frac{|\Phi(\xi_j)|}{\rho_{m^{*}}(F)}<1.
\end{equation}
Let $\varepsilon>0.$ By Cauchy's integral formula, we have for $j =1,\ldots,w$ and $t =0,1,\ldots,\tau_j-1,$
\begin{equation}\label{tfvjkml}
(Q_{n,m,m^{*}}^{\mu})^{(t)}(\xi_j)-Q_{m}^{(t)}(\xi_j)=\frac{t!}{2\pi i} \int_{|z-\xi_j|=\varepsilon} \frac{Q_{n,m,m^{*}}^{\mu}(z)-Q_{m}(z)}{(z-\xi_j)^{t+1}}dz.
\end{equation}
Using \eqref{3.3332} and \eqref{3.32new}, it follows from \eqref{tfvjkml} that for $j =1,\ldots,w$ and $t =0,1,\ldots,\tau_j-1,$
$$\limsup_{n \rightarrow \infty}|Q_{m}^{(t)}(\xi_j)|^{1/n}<1$$
and $Q_{m}^{(t)}(\xi_j)=0,$
 which means that for each $j =1,\ldots,w,$ the order of the zero $\xi_j$ of $Q_{m}$ is at least $\tau_j.$ Hence, $Q_{m}$ can be divided by $q_{m^{*}}.$ This implies
$\rho_{0}(Q_{m} F)\geq \rho_{m^{*}}(F).$
\end{proof}

The following result is used in Lemma \ref{mainlemma}. For the proof see \cite[Lemma 3]{bosuwan1}.

\begin{lemma}\label{trick}
 Let $N_0\in \mathbb{N}$ and $C>0.$ If a sequence of complex numbers $\{F_N\}_{N \in \mathbb{N}}$ has the following properties:
\begin{enumerate}
\item[$(i)$] $\lim_{N \rightarrow \infty} |F_N|^{1/N}=0,$
\item[$(ii)$] $|F_N|\leq C \sum_{k=N+1}^{\infty} |F_k|,$ for all $N \geq N_0$,
\end{enumerate}
then there exists $N_1\in \mathbb{N}$ such that $F_N=0$ for all $N \geq N_1.$
\end{lemma}

The next lemma is the cornerstone for obtaining the inverse statements contained Theorem \ref{inverse}.
\begin{lemma}\label{mainlemma} Let $F\in \mathcal{H}(E)$ and $\mu\in \textup{\textbf{Reg}}_{1,2}^{m}(E).$ Fix $m\geq m^{*}\geq 1.$  Suppose that $F$ is not a rational function with at most $m^{*}-1$ poles and there exists a polynomial $Q_{m}$ of degree $m$ such that
\begin{equation}\label{3.33321}
\limsup_{n \rightarrow \infty} \|Q_{n,m,m^{*}}^{\mu}-Q_{m}\|^{1/n}=\theta<1.
\end{equation}
Then, either $F$ has exactly $m^{*}$ poles in $D_{\rho_{m^{*}}(F)}$ or $\rho_0(Q_{m} F)>\rho_{m^{*}}(F).$
\end{lemma}
\begin{proof}[Proof of Lemma \ref{mainlemma}] From Lemma \ref{lemma4}, we know that $\rho_{0}(Q_{m} F)\geq \rho_{m^{*}}(F).$ Assume that $\rho_{0}(Q_{m} F)=\rho_{m^{*}}(F).$ Let us show that $F$ has exactly $m^{*}$ poles in $D_{\rho_{m^{*}}(F)}.$
To the contrary, suppose that $F$ has in $D_{\rho_{m^{*}}(F)}$ at most $m^{*}-1$ poles. Then, there exists a polynomial $q_{m^{*}}$ with $\deg q_{m^{*}}< m^{*}$ such that
$$\rho_0(q_{m^{*}} F)=\rho_{m^{*}}(F)=\rho_0(Q_{m}q_{m^{*}} F).$$ Since $\deg q_{m^{*}}< m^{*},$ by the definition of $Q_{n,m,m^{*}}^{\mu},$ $[Q_{n,m,m^{*}}^{\mu} q_{m^{*}} F]_n=0.$ Take $1< \rho< \rho_{m^{*}}(F).$ Then, by Lemma \ref{expan} and \eqref{usealot},
$$\frac{1}{\rho_{m^{*}}(F)}=\limsup_{n \rightarrow \infty} |[Q_{m}q_{m^{*}} F]_n|^{1/n}=\limsup_{n \rightarrow \infty} |[Q_{m}q_{m^{*}} F-Q_{n,m,m^{*}}^{\mu} q_{m^{*}} F]_n|^{1/n}$$
$$=\limsup_{n \rightarrow \infty} \left| \frac{1}{2\pi i}  \int_{\Gamma_{\rho}} (Q_{m}-Q_{n,m,m^{*}}^{\mu})(z)  q_{m^{*}}(z) F(z) s_n(z)   dz\right|^{1/n}.$$
From the above relation, using \eqref{asymps} and \eqref{3.33321}, it is easy to show that
$$\frac{1}{\rho_{m^{*}}(F)}\leq \frac{\theta}{\rho_{m^{*}}(F)},$$ which is possible only if $\rho_{m^{*}}(F)=\infty.$

Now, let us show that this is impossible. From \eqref{3.33321}, without loss of generality, we can assume that $\deg Q_{n,m,m^{*}}^{\mu}=m.$ Set
$$q_{m^{*}}(z)F(z) :=\sum_{k=0}^{\infty} a_k p_k(z)$$
and
$$Q_{n,m,m^{*}}^{\mu}(z):=\sum_{j=0}^{m}b_{n,j}z^{j},$$
where $b_{n,m}=1.$ From  \eqref{3.33321}, there exists $n_1\in \mathbb{N},$
\begin{equation}\label{suppart}
\sup\{|b_{n,j}|: 0\leq j \leq m, \, n \geq n_1\}\leq c_{1}.
\end{equation}
 Recall that $[Q_{n,m,m^{*}}^{\mu} q_{m^{*}} F]_n=0.$ Therefore,
$$0=[Q_{n,m,m^{*}}^{\mu} q_{m^{*}} F]_n=\sum_{k=0}^{\infty} \sum_{j=0}^{m} a_k b_{n,j} [ z^{j} p_k]_n=\sum_{k=n-m}^{\infty} \sum_{j=0}^{m} a_k b_{n,j} [ z^{j} p_k]_n$$
\begin{equation}\label{almostdone}
=\sum_{k=n-m}^{\infty} \sum_{j=0}^{m} a_k b_{n,j} \langle z^{j} p_k, \, p_n\rangle_{\mu}=\frac{\kappa_{n-m}}{\kappa_{n}} a_{n-m}+\sum_{k=n-m+1}^{\infty} \sum_{j=0}^{m} a_k b_{n,j} \langle z^{j} p_k, \, p_n\rangle_{\mu}.
\end{equation}
By the Cauchy-Schwarz inequality and the orthonormality of $p_{n},$ for all $k,n\geq 0$ and $j\in \{1,\ldots,m\},$
\begin{equation}\label{CauchySchwarz}
|\langle z^{j} p_k, \, p_n\rangle_{\mu}|\leq \|z^{j}\|_E  \langle p_k, \, p_k\rangle_{\mu}^{1/2} \langle  p_n, \, p_n\rangle_{\mu}^{1/2}= \|z^{j}\|_E\leq c_{2}.
\end{equation}
By \eqref{below}, there exists $n_2\in \mathbb{N}$ such that for all $n \geq n_2,$
\begin{equation}\label{capacity}
\frac{\kappa_{n-m}}{\kappa_{n}}\geq c_{3}>0.
\end{equation}
Combining \eqref{suppart}, \eqref{CauchySchwarz}, and \eqref{capacity}, it follows from \eqref{almostdone} that
$$|a_{n-m}|\leq c_{4}\sum_{k=n-m+1}^{\infty } |a_{k}|.$$
 Setting $N=n-m,$ we obtain
 $$|a_{N}|\leq c_{4}\sum_{k=N+1}^{\infty } |a_{k}|.$$ By Lemma \ref{trick}, since $\lim_{N \rightarrow \infty} |a_N|^{1/N}=0,$ there exists $N_1\in \mathbb{N}$ such that $a_{N}=0$ for all $N \geq N_1,$ which implies that $q_{m^{*}}F$ is a polynomial and $F$ is a rational function with at most $m^{*}-1$ poles. This contradicts the assumption that $F$ is not a rational function with at most $m^{*}-1$ poles. Then, $F$ has exactly $m^{*}$ poles in $D_{\rho_{m^{*}}(F)}.$
\end{proof}

\subsection{Polynomial independence}

Let us introduce the concept of polynomial independence of a vector of functions.
\begin{definition}\label{polydef}\textup{
A vector $\textup{\textbf{F}}\in \mathcal{H}(E)^d$ is said to be \emph{polynomially independent with respect to $\textup{\textbf{m}}=(m_1,\ldots,m_d)\in \mathbb{N}^d$} if there do not exist polynomials $v_1,\ldots,v_d,$ at least one of which is non-null, such that
\begin{enumerate}
\item [(i)] $\deg v_i < m_i,$ $i=1,\ldots,d,$
\item [(ii)] $\sum_{i=1}^d v_i f_i$ is a polynomial.
\end{enumerate}}
\end{definition}
Note that if $\textup{\textbf{F}}$ is polynomially independent, then for each $i=1,\ldots,d,$ $F_i$ is not a rational function with at most $m_i-1$ poles.

The following lemma  reduces the use of  polynomial combinations in \eqref{polycom} to that of  linear combinations.

\begin{lemma}\label{1reduce} Let $\textup{\textbf{F}}\in \mathcal{H}(E)^d,$ $\mu\in\mathcal{M}(E),$ and fix a multi-index $\textup{\textbf{m}}\in \mathbb{N}^d.$ Set
\begin{equation}\label{definefunction}
\overline{\textup{\textbf{F}}}:=(F_1,\ldots,z^{m_1-1}F_1, F_2,\ldots, z^{m_d-1}F_d)=(f_1,f_2,\ldots,f_{|\textup{\textbf{m}}|})
\end{equation}
and define an associated multi-index
\begin{equation}\label{definemulti}
\overline{\textup{\textbf{m}}}:=(1,1,\ldots,1)
\end{equation} with $|\overline{\textup{\textbf{m}}}|=|\textup{\textbf{m}}|.$ Then:
\begin{enumerate}
\item [(i)] $\textup{\textbf{F}}$ is polynomially independent with respect to $\textup{\textbf{m}}$ if and only if $\overline{\textup{\textbf{F}}}$ is polynomially independent with respect to $\overline{\textup{\textbf{m}}}.$
\item [(ii)] the poles of $\textup{\textbf{F}}$ and their orders are the same as the poles of $\overline{\textup{\textbf{F}}}$ and their orders.
\item [(iii)] $\rho_{m}(\textup{\textbf{F}})=\rho_{m}(\overline{\textup{\textbf{F}}}),$ for all $m \in \mathbb{N}\cup\{0\}.$
\item [(iv)] the systems of equations that define $Q_{n,|\textup{\textbf{m}}|}^{\mu}$ for $\textup{\textbf{F}}$ and  $\textup{\textbf{m}}$, and $Q_{n,|\overline{\textup{\textbf{m}}}|}^{\mu}$ for $\overline{\textup{\textbf{F}}}$ and $\overline{\textup{\textbf{m}}}$ are the same.
\item [(v)]  the system poles of $\textup{\textbf{F}}$ with respect to $\textup{\textbf{m}}$ and their orders  are the same as the system poles  of $\overline{\textup{\textbf{F}}}$ with respect to $\overline{\textup{\textbf{m}}}$ and their orders.
\end{enumerate}
\end{lemma}
The proof of the previous lemma is straightforward and we leave it to the reader.

\begin{lemma}\label{polylem} Let $\textup{\textbf{F}}\in \mathcal{H}(E)^d,$ $\mu\in \mathcal{M}(E),$ and fix a multi-index $\textup{\textbf{m}}\in \mathbb{N}^d.$ Suppose that for all $n \geq n_0,$ the polynomial $Q_{n,|\textup{\textbf{m}}|}^{\mu}$ is unique and $\deg Q_{n,|\textup{\textbf{m}}|}^{\mu}=|\textup{\textbf{m}}|.$ Then the system $\textup{\textbf{F}}$ is polynomially independent with respect to $\textup{\textbf{m}}.$
\end{lemma}
\begin{proof}[Proof of Lemma \ref{polylem}] From Lemma \ref{1reduce}, without loss of generality, we consider $\overline{\textup{\textbf{F}}}$ as defined in \eqref{definefunction}  and $\overline{\textup{\textbf{m}}}$ as defined in \eqref{definemulti}. Notice that $Q_{n,|\textup{\textbf{m}}|}^{\mu}=Q_{n,|\overline{\textup{\textbf{m}}}|}^{\mu}.$

Suppose that there exist $c_i,$ $i=1,\ldots,|\overline{\textup{\textbf{m}}}|,$ such that $\sum_{i=1}^{|\overline{\textup{\textbf{m}}}|} c_i f_i$ is a polynomial, say $q.$ Without loss of generality, we can assume that $c_1\not=0.$ Then,
$$f_1=p-\sum_{i=2}^{|\overline{\textup{\textbf{m}}}|} c_i' f_i,$$
where $c_i':=c_i/c_1$ and we denote by $N$ the degree of $p=q/c_1.$

On the other hand, the homogenoeous system of linear equations
$$\langle Q_{n} f_i,\, p_{n} \rangle_{\mu}=0,  \quad \quad i=2,\ldots,|\overline{\textup{\textbf{m}}}|,$$
where $\deg(Q_{n})\leq |\overline{\textup{\textbf{m}}}|-1, Q_{n}\not\equiv 0,$ has a solution, say a monic polynomial $\tilde{Q}_n.$ Moreover, for $n\geq |\overline{\textup{\textbf{m}}}|+N,$
$$\langle \tilde{Q}_{n} f_1,\, p_{n} \rangle_{\mu}=[\tilde{Q}_{n}f_1]_n=[\tilde{Q}_{n}p-\sum_{i=2}^{|\overline{\textup{\textbf{m}}}|} c_i'\tilde{Q}_{n}  f_i]_n$$
$$=[\tilde{Q}_{n}p]_n-\sum_{i=2}^{|\overline{\textup{\textbf{m}}}|} c_i'[\tilde{Q}_{n}  f_i]_n=0,$$ which means $\tilde{Q}_n=Q_{n,|\textup{\textbf{m}}|}^{\mu}.$ However,  $\deg(\tilde{Q}_{n})\leq |\overline{\textup{\textbf{m}}}|-1$ which contradicts our assumption on $Q_{n,|\textup{\textbf{m}}|}^{\mu}.$ This completes the proof.
\end{proof}

A direct consequence of Lemma \ref{mainlemma} is the following.
\begin{lemma}\label{mainlemma1} Let $\textup{\textbf{m}}=(m_1,\ldots,m_d)\in \mathbb{N}^d$ be a fixed multi-index, $\mu\in \textup{\textbf{Reg}}_{1,2}^{|\textup{\textbf{m}}|}(E),$ and  $\textup{\textbf{F}}\in \mathcal{H}(E)^d.$  Suppose that $\textup{\textbf{F}}$ is polynomially independent with respect to $\textup{\textbf{m}}$ and there exists a polynomial $Q_{|\textup{\textbf{m}}|}$ of degree $|\textup{\textbf{m}}|$ such that
\begin{equation}\label{3.333211}
\limsup_{n \rightarrow \infty} \|Q_{n,|\textup{\textbf{m}}|}^{\mu}-Q_{|\textup{\textbf{m}}|}\|^{1/n}=\theta<1.
\end{equation}
Then, for each $i=1,\ldots,d,$ either $F_i$ has exactly $m_i$ poles in $D_{\rho_{m_i}(F_i)}$ or $\rho_0(Q_{|\textup{\textbf{m}}|} F_i)>\rho_{m_i}(F_i).$
\end{lemma}

Lemma \ref{numberof} below contains some straightforward consequences of the concept of system poles. Its proof is analogous to that of \cite[Lemma 3.5]{CacoqYsernLopez} so we will not dwell into details.

\begin{lemma}\label{numberof}
Given $\textup{\textbf{F}}\in \mathcal{H}(E)^d$ and $\textup{\textbf{m}}\in \mathbb{N}^d,$ $\textup{\textbf{F}}$ can have at most $|\textup{\textbf{m}}|$ system poles with respect to $\textup{\textbf{m}}$ (counting their order). Moreover, if the system $\textup{\textbf{F}}$ has exactly $|\textup{\textbf{m}}|$ system poles with respect to $\textup{\textbf{m}}$ and $\xi$ is a system pole of order $\tau,$ then for all $s>\tau$ there can be no polynomial combination of the form \eqref{polycom} holomorphic on a neighborhood of $\overline{D}_{|\Phi(\xi)|}$ except for a pole at $z=\xi$ of exact order $s.$
\end{lemma}

\subsection{Proof of (b) implies (a)}

The following theorem contains the inverse statements.
\begin{thm}\label{inverse} Let $\textup{\textbf{F}}=(F_1,F_2,\ldots,F_d)\in \mathcal{H}(E)^d,$  $\textup{\textbf{m}}\in \mathbb{N}^d,$ be a fixed multi-index, and $\mu\in \textup{\textbf{Reg}}_{1,2}^{|\textup{\textbf{m}}|}(E).$  Suppose that the denominators $Q_{n,|\textup{\textbf{m}}|}^{\mu}$ of the orthogonal Hermite-Pad\'e approximants are uniquely determined for all sufficiently large $n,$ and there exists a polynomial $Q_{|\textup{\textbf{m}}|}$ of degree $|\textup{\textbf{m}}|$ such that
$$\limsup_{n \rightarrow \infty} \|Q_{n,|\textup{\textbf{m}}|}^{\mu}-Q_{|\textup{\textbf{m}}|}\|^{1/n}=\theta<1.$$
Then, $\textup{\textbf{F}}$ has exactly $|\textup{\textbf{m}}|$ system poles with respect to $\textup{\textbf{m}}$ counting multiplicities and $Q_{|\textup{\textbf{m}}|}=Q_{\textup{\textbf{m}}}^{\textup{\textbf{F}}}.$
\end{thm}

\begin{proof}[Proof of Theorem \ref{inverse}] Due to Lemma \ref{1reduce}, without loss of generality, we can restrict our attention to the vector of functions
$$\overline{\textup{\textbf{F}}}:=(F_1,\ldots,z^{m_1-1}F_1, F_2,\ldots, z^{m_d-1}F_d)=(f_1,f_2,\ldots,f_{|\textup{\textbf{m}}|})
$$ and the associated multi-index
$$
\overline{\textup{\textbf{m}}}:=(1,1,\ldots,1)
$$
with $|\overline{\textup{\textbf{m}}}|=|\textup{\textbf{m}}|.$ Notice that $Q_{n,|\textup{\textbf{m}}|}^{\mu}=Q_{n,|\overline{\textup{\textbf{m}}}|}^{\mu}.$ Moreover, due to Lemma \ref{polylem}, we know that $\textup{\textbf{F}}$ is polynomially independent with respect to $\textup{\textbf{m}}$ which implies that $\overline{\textup{\textbf{F}}}$ is polynomially independent with respect to $\overline{\textup{\textbf{m}}}$ according to Lemma \ref{1reduce}.

The auxiliary results that we have obtained thus far allow us to mimic the proof employed in \cite{CacoqYsernLopez} to obtain a similar result for the case of row sequences of Hermite-Pad\'e approximations. For completeness we include the whole proof. The plan is as follows. First, we collect a set of $|\textup{\textbf{m}}|$ candidates to be system poles of $\overline{\textup{\textbf{F}}}$ (counting their orders) and prove that they are the zeros of $Q_{|\textup{\textbf{m}}|}.$ Secondly, we prove that all these points previously selected are actually system poles of $\overline{\textup{\textbf{F}}}$ which means that they are also system poles of $\textup{\textbf{F}}$ by Lemma \ref{1reduce}.

From Lemma \ref{mainlemma1}, for each $i=1,\ldots,|\textup{\textbf{m}}|,$ either $D_{\rho_1(f_i)}$ contains exactly one pole of $f_i$ and it is a zero of $Q_{|\textup{\textbf{m}}|},$ or $\rho_0(Q_{|\textup{\textbf{m}}|} f_i)>\rho_{1}(f_i).$ Hence, $D_{\rho_{0}(\overline{\textup{\textbf{F}}})}\not=\mathbb{C}$ and $Q_{|\textup{\textbf{m}}|}$ contains as zeros all the poles of $f_i$ on the boundary of $D_{\rho_{0}(f_i)}$ counting their order for $i=1,\ldots,|\textup{\textbf{m}}|.$ Moreover, the function $f_i$ cannot have on the boundary of $D_{0}(f_i)$ singularities other than poles. Hence, the poles of $\overline{\textup{\textbf{F}}}$ on the boundary of $D_{\rho_0(\overline{\textup{\textbf{F}}})}$ are all zeros of $Q_{|\textup{\textbf{m}}|}$ counting multiplicities and the boundary contains no other singularity except poles. Let us call them candidate system poles of $\overline{\textup{\textbf{F}}}$ and denote them by $a_1,\ldots,a_{n_1}$ repeated according to their order. They constitute the first layer of candidate system poles of $\overline{\textup{\textbf{F}}}.$

Since $\deg Q_{|\textup{\textbf{m}}|}=|\textup{\textbf{m}}|,$ $n_1\leq |\textup{\textbf{m}}|.$ If $n_1=|\textup{\textbf{m}}|,$ we are done. Let us assume that $n_1< |\textup{\textbf{m}}|$ and find coefficients $c_1,\ldots,c_{|\textup{\textbf{m}}|}$ such that
$$\sum_{i=1}^{|\textup{\textbf{m}}|} c_i f_i$$
is holomorphic in a neighborhood of $\overline{D}_{\rho_0(\overline{\textup{\textbf{F}}})}.$ Finding those $c_1,\ldots,c_{|\textup{\textbf{m}}|}$ reduces to solving a homogeneous system of $n_1$ linear equations with $|\textup{\textbf{m}}|$ unknowns. In fact, if $z=a$ is a candidate system pole of $\overline{\textup{\textbf{F}}}$ with multiplicity $\tau,$ we obtain $\tau$ equations choosing the coefficients $c_i$ so that
\begin{equation}\label{eq21}
\int_{|w-a|=\delta} (w-a)^{k} \left( \sum_{i=1}^{|\textup{\textbf{m}}|} c_i f_i(w)\right)dw=0, \quad \quad k=0,\ldots,\tau-1.
\end{equation}
We have the same type of  system of equations for each distinct candidate system pole on the boundary of $D_{\rho_0(\overline{\textup{\textbf{F}}})}.$ Combining these equations, we obtain a homogeneous system of $n_1$ linear equations with $|\textup{\textbf{m}}|$ unknowns. Moreover, this homogeneous system of linear equations has at least $|\textup{\textbf{m}}|-n_1$ linearly independent solutions, which we denote by $\textup{\textbf{c}}_{j}^{1},$ $j=1,\ldots, |\textup{\textbf{m}}|-n_1^{*},$ where $n_1^{*}\leq n_1$ denotes the rank of the system of equations.

Let
$$\textup{\textbf{c}}_{j}^{1}:=(c_{j,1}^1,\ldots, c_{j,|{\textup{\textbf{m}}}|}^{1}), \quad \quad j=1,\ldots,  |{\textup{\textbf{m}}}|-n_1^{*}.$$ Define the $(|{\textup{\textbf{m}}}|-n_1^{*})\times |{\textup{\textbf{m}}}|$ dimensional matrix
$$C^1:= \begin{pmatrix}
 \textup{\textbf{c}}_{1}^{1}  \\
  \vdots   \\
\textup{\textbf{c}}_{|{\textup{\textbf{m}}}|-n_1^{*}}^{1}
 \end{pmatrix}.$$
Define the vector  $\textup{\textbf{g}}_1$ of $|{\textup{\textbf{m}}}|-n_1^{*}$ functions given by
 $$\textup{\textbf{g}}_1^{t}:=C^1 \overline{\textup{\textbf{F}}}^t=(g_{1,1},\ldots,g_{1, |{\textup{\textbf{m}}}|-n_1^{*}})^{t},$$
 where $A^t$ denotes the transpose of the matrix $A.$ Since all the rows of $C^1$ are non-null and $\overline{\textup{\textbf{F}}}$ is polynomially independent with respect to $\overline{\textup{\textbf{m}}},$ none of the functions
 $$g_{1,j}=\sum_{i=1}^{|{\textup{\textbf{m}}}|} c_{j,i}^{1} f_i, \quad \quad j=1,\ldots,|{\textup{\textbf{m}}}|-n_1^{*},$$
 are  polynomials.

 Consider the canonical domain
 $$D_{\rho_0(\textup{\textbf{g}}_1)}=\bigcap_{j=1}^{|{\textup{\textbf{m}}}|-n_1^{*}} D_{\rho_0(g_{1,j})}.$$ Clearly, $ D_{\rho_0(\overline{\textup{\textbf{F}}})}\subsetneq D_{\rho_0(\textup{\textbf{g}}_1)}$ and $[Q_{n,|\overline{\textup{\textbf{m}}}|}^{\mu}g_{1,j}]_n=0$ for all  $j=1,\ldots,|{\textup{\textbf{m}}}|-n_1^{*}.$ Therefore, for each $j=1,\ldots,|\textup{\textbf{m}}|-n_1^{*},$ $Q_{n,|\overline{\textup{\textbf{m}}}|}^{\mu}$ is a denominator of an $(n,|\overline{\textup{\textbf{m}}}|,1)$ incomplete orthogonal Pad\'e approximant of $g_{1,j}$ with respect to $\mu$. Since all $g_{1,j}$ are not polynomials, by Lemma \ref{mainlemma} with $m^{*}=1,$ for each $j=1,\ldots,|{\textup{\textbf{m}}}|-n_1^{*},$ either $D_{\rho_1(g_{1,j})}$ contains exactly one pole of $g_{1,j}$ and it is a zero of $Q_{|\textup{\textbf{m}}|},$ or $\rho_{0}(Q_{|\textup{\textbf{m}}|}g_{1,j})>\rho_{1}(g_{1,j}).$ In particular, $D_{\rho_0(\textup{\textbf{g}}_1)}\not=\mathbb{C}$ and all the singularities of $\textup{\textbf{g}}_1$ on the boundary of $D_{\rho_0(\textup{\textbf{g}}_1)}$ are poles which are zeros of $Q_{|\textup{\textbf{m}}|}$ counting their order. They form the next layer of candidate system poles of $\overline{\textup{\textbf{F}}}.$

 Denote by $a_{n_1+1},\ldots,a_{n_1+n_2}$ these new candidate system poles. Again, if $n_1+n_2=|\textup{\textbf{m}}|,$ we are done. Otherwise, $n_2<|\textup{\textbf{m}}|-n_1\leq |\textup{\textbf{m}}|-n_{1}^{*},$ and we keep repeating the same process by eliminating the $n_2$ poles $a_{n_1+1},\ldots,a_{n_1+n_2}.$ In order to do that, we have $|\textup{\textbf{m}}|-n_{1}^{*}$ functions which are holomorphic on $D_{\rho_0(\textup{\textbf{g}}_1)}$ and meromorphic on a neighborhood of $\overline{D}_{\rho_{0}(\textup{\textbf{g}}_1)}.$ The corresponding homogeneous system of linear equations, similar to \eqref{eq21}, has at least $|\textup{\textbf{m}}|-n_1^{*}-n_2$ linearly independent solutions $\textup{\textbf{c}}_j^2,$ $j=1,\ldots,|\textup{\textbf{m}}|-n_1^{*}-n_2^{*},$ where $n_2^{*}\leq n_2$ is the rank of the new system. Let
$$\textup{\textbf{c}}_{j}^{2}:=(c_{j,1}^2,\ldots, c_{j,|{\textup{\textbf{m}}}|}^{2}), \quad \quad j=1,\ldots,  |{\textup{\textbf{m}}}|-n_1^{*}-n_2^{*}.$$ Define the $(|{\textup{\textbf{m}}}|-n_1^{*}-n_2^{*})\times (|{\textup{\textbf{m}}}|-n_1^{*})$ dimensional matrix
$$C^2:= \begin{pmatrix}
 \textup{\textbf{c}}_{1}^{2}  \\
  \vdots   \\
\textup{\textbf{c}}_{|{\textup{\textbf{m}}}|-n_1^{*}-n_2^{*}}^{2}
 \end{pmatrix}.$$
 Define the vector  $\textup{\textbf{g}}_2$ of  $|{\textup{\textbf{m}}}|-n_1^{*}-n_2^{*}$ functions given by
 $$\textup{\textbf{g}}_2^{t}:=C^2\textup{\textbf{g}}_1^t=C^2C^1 \overline{\textup{\textbf{F}}}^t=(g_{2,1},\ldots,g_{2, |{\textup{\textbf{m}}}|-n_1^{*}-n_2^{*}})^{t}.$$
It is a basic fact from linear algebra that if $C_1$ has full rank and $C_2$ has  full rank, then $C_2 C_1$ has full rank. This means that the rows of $C_2 C_1$ are linearly independent, particularly, they are non-null. Therefore, none of the component functions of $\textup{\textbf{g}}_2$ are polynomials because of the polynomial independence of $\overline{\textup{\textbf{F}}}$ with respect to $\overline{\textup{\textbf{m}}}.$ Thus, we can apply again Lemma \ref{mainlemma}. Using finite induction, we find  a total on $|\overline{\textup{\textbf{m}}}|$ candidate system poles.

In fact, on each layer of system poles, $ n_i\geq 1.$ Therefore, in a finite number of steps, say $N-1,$ their sum equals $|\textup{\textbf{m}}|.$   Consequently, the number of candidate system poles of $\overline{\textup{\textbf{F}}}$ in some disk, counting multiplicities, is exactly equal to $|\textup{\textbf{m}}|,$ and they are precisely the zeros of $Q_{|\textup{\textbf{m}}|}$ as we wanted to prove.

Summarizing, in the $N-1$ steps we have taken, we have produced $N$ layers of candidate system poles. Each layer contains $n_k$ candidates, $k=1,\ldots,N.$ At the same time, on each step $k,$ $k=1,\ldots,N-1,$ we have solved a system of $n_k$  linear equations, of rank $n_k^{*},$ with $|\textup{\textbf{m}}|-n_1^{*}-\cdots-n_k^{*},$ $n_k^{*}\leq n_k,$ linearly independent solutions. We find ourselves on the $N$-th layer with $n_N$ candidates.

Let us try to eliminate the poles on the last layer. Write the corresponding homogeneous system of linear equations as in \eqref{eq21}, and we get $n_N$ equations where
$$n_N=|\textup{\textbf{m}}|-n_1-\cdots-n_{N-1}\leq |\textup{\textbf{m}}|-n_1^{*}-\cdots-n_{N-1}^{*}=:\overline{n}_N$$
 with $\overline{n}_N$ unknowns. For each candidate system pole $a$ of multiplicity $\tau$ on the $N$-th layer, we impose the equations
\begin{equation}\label{eq22}
\int_{|w-a|=\delta} (w-a)^i\left(\sum_{k=1}^{\overline{n}_N} c_k g_{N-1,k}(w) \right)dw=0,\quad \quad i=0,\ldots,\tau-1,
\end{equation}
where $\delta$ is sufficiently small and the $g_{N-1,k},$ $k=1,\ldots,\overline{n}_N,$ are the functions associated with the linearly independent solutions produced on step $N-1.$

Let $n_N^{*}$ be the rank of this last homogeneous system of linear equations. Assume that $n_k^{*}<n_k$ for some $k\in \{1,\ldots,N\}.$ Then, the rank of the last system of equations is strictly less than the number of unknowns, namely $n_N^{*}<\overline{n}_N.$ Therefore, repeating the same process, there exists a vector of functions $$\textup{\textbf{g}}_N:=(g_{N,1},\ldots, g_{N,|\textup{\textbf{m}}|-n_1^{*}-\cdots-n_{N}^{*}})$$ such that none of the $g_{N,k}$ is a polynomial because of the polynomial independence of $\overline{\textup{\textbf{F}}}$ with respect to $\overline{\textup{\textbf{m}}}.$ Applying Lemma \ref{mainlemma}, each $g_{N,k}$ has on the boundary of its disk of analyticity a pole which is a zero of $Q_{|\textup{\textbf{m}}|}.$ However, this is impossible because all the zeros of $Q_{|\textup{\textbf{m}}|}$ are strictly contained in that disk. Consequently, $n_k=n_k^{*},$ $k=1,\ldots,N.$

We conclude that all the $N$ homogeneous systems of linear equations that we have solved  have full rank. This implies that if in any one of those $N$ systems of equations we equate one of its equations to $1$ instead of zero (see \eqref{eq21} or \eqref{eq22}), the corresponding nonhomogeneous system of linear equations has a solution. By the definition of a system pole, this implies that each candidate system pole is indeed a system pole of order at least equal to its multiplicity as zero of $Q_{|\textup{\textbf{m}}|}.$ Moreover, by Lemma \ref{numberof}, $\overline{\textup{\textbf{F}}}$ can have at most $|\textup{\textbf{m}}|$ system poles with respect to $\overline{\textup{\textbf{m}}}$; therefore, all candidate system poles are system poles, and their order coincides with the multiplicity of that point as a zero of $Q_{|\textup{\textbf{m}}|}.$ This also means that $Q_{|\textup{\textbf{m}}|}=Q_{\textup{\textbf{m}}}^{\textup{\textbf{F}}}.$ We have completed the proof.
\end{proof}

\begin{remark} The results of this paper remain valid when $E$ is a compact set whose complement is connected, provided that the sequence of orthonormal polynomials $(p_n), n\geq 0,$ and second type functions $(s_n),n\geq 0,$ relative to the measure $\mu, \mbox{supp}(\mu) \subset E,$ satisfy \eqref{asintlog} and \eqref{asintlog2}, respectively, inside $\mathbb{C} \setminus E$, and the sequence of leading coefficients $(\kappa_n), n\geq 0,$ fulfill \eqref{below}. On the right hand side of  \eqref{asintlog} and \eqref{asintlog2} one should place $e^{g(z,\infty)},$ where $g(z,\infty)$ denotes Green's function relative to the region $\mathbb{C} \setminus E$. The problem with stating the results with this degree of generality is related with the zeros that the second type functions $s_n$ may have in $\mbox{\rm Co}(E) \setminus E$. For example, if $E$ is made up of two intervals symmetric with respect to the origin and $\mu$ is any measure supported on $E$ symmetric with respect to the origin then  $s_n$ has a zeros at $z=0$ for all even $n$. In this case, no matter how good the measure is, there are problems in proving $\eqref{3.31}$ at $\xi=0$ if this point happens to be a system pole. In this example, this can be avoided requiring that $0$ is not a system pole of ${\bf F}$. But, in a more general configuration, this is hard to guarantee in terms of the data since the zeros of $s_n$ in  $\mbox{\rm Co}(E) \setminus E$ may have a quite exotic behavior.
\end{remark}

\noindent Nattapong Bosuwan\\
%{first address}
         Department of Mathematics,
         Mahidol University\\
         Rama VI Road, Ratchathewi District,\\
         Bangkok 10400, Thailand\\
and \\
% {second address}
 \noindent Centre of Excellence in Mathematics, CHE,\\
                 Si Ayutthaya Road,\\
                 Bangkok 10400, Thailand\\
\noindent email: nattapong.bos@mahidol.ac.th,\\

\noindent\\
Guillermo L\'opez Lagomasino \\
Mathematics Department,\\
Universidad Carlos III de Madrid, \\
c/ Universidad, 30\\
28911, Legan\'es, Spain \\
email: lago@math.uc3m.es

\end{document}